\documentclass[]{article}
\usepackage{amsmath, amssymb}
\usepackage{enumerate}
\usepackage{enumitem}
\usepackage{graphicx}
\usepackage{float}
\usepackage{hyperref}
\usepackage{verbatim}
\usepackage{amssymb}
\usepackage{color}
\usepackage{amsthm}
\usepackage{amssymb}
\usepackage{verbatim}
\usepackage{geometry}
\usepackage{comment}
\usepackage{mathtools}
\usepackage{mathrsfs}
\usepackage{float}
\usepackage{subcaption}
\usepackage{tikz} 
\usetikzlibrary{angles, quotes}
\usetikzlibrary{shapes.geometric}% Permet d'utiliser tikz pour faire des graphiques

\def\R{{\mathbb R}}
\def\N{{\mathbb N}}

\def\C{{\mathbb C}}

\def\ds{\displaystyle}

\def\dt{{\rm d}t}

\def\dx{{\rm d}x}

\newcommand{\norm}[1]{\left\Vert#1\right\Vert}
\newcommand{\ud}{{\rm{d}}}
\newcommand{\Supp}{{\rm{Supp\,}}}
\renewcommand{\leq}{\leqslant}                 
\renewcommand{\geq}{\geqslant}
\renewcommand{\tilde}{\widetilde}

\renewcommand{\div}{\text{div\,}}
 \newtheorem{theorem}{Theorem}[section]
 \newtheorem{remark}[theorem]{Remark}
  \newtheorem{Assumption}{Assumption}[section]

 \newtheorem{corollary}[theorem]{Corollary}

 \newtheorem{proposition}[theorem]{Proposition}
 \newtheorem{definition}[theorem]{Definition}
 \numberwithin{equation}{section}

\title{
Observability of wave and plate equations with rough coefficients and interfaces: a multiplier approach
\footnote{Financial support. S.E. is partially supported by the ANR projects TRECOS ANR 20-CE40-0009, NumOpTes ANR-22-CE46-0005 and CHAT ANR-24-CE40-5470, and the MathAmSud project SCIPinPDEs 51749UC. Part of this work was done while S.E. was visiting Universidad Técnica Federico Santa Mar\'{\i}a, and S.E. acknowledges this university for providing a nice and fruitful working environment. A.M. is partially supported by Proyecto Interno USM 2025 PI-LIR-25-14. A. O. was partially funded by ANID-Fondecyt 1240200, 1231404, 
CMM FB210005 Basal-ANID, FONDAP/1523A0002, FONDEF IT23I0095, 
ECOS 240038 and DO ANID Technology Center DO210001.}}

\author{
	Sylvain \textsc{Ervedoza}\footnote{Institut de Mathématiques de Bordeaux, UMR 5251, Université de Bordeaux, CNRS, Bordeaux INP, F-33400 Talence, France,  
\texttt{sylvain.ervedoza@math.u-bordeaux.fr}}, 
	\and
	Alex \textsc{Imba} \footnote{Universidad San Francisco de Quito USFQ, Colegio de Ciencias e Ingenier\'ias, Campus Cumbayá, Casilla Postal 17-1200-841, Quito 170901, Ecuador, \texttt{aimba@usfq.edu.ec}} 
	\and
	Alberto \textsc{Mercado}\footnote{Departamento de Matemática, Universidad Técnica Federico Santa Mar\'{\i}a, Chile, \texttt{alberto.mercado@usm.cl}} 
	\and
	Axel \textsc{Osses}\footnote{Departamento de Ingeniería Matemática and Centro de Modelamiento Matemático (UMI 2807 CNRS), FCFM Universidad de Chile, Casilla 170/3 - Correo 3, Santiago, Chile, \texttt{axosses@dim.uchile.cl}} 
}

\begin{document}

\maketitle

\begin{abstract}
The goal of this article is to derive observability properties for wave and plate equations with rough coefficients in the principal part, including the case of interfaces, through a regularization process. Specifically, we demonstrate that if a singular coefficient can be approximated by a sequence of smooth coefficients corresponding to uniformly observable systems, then the observability inequalities can be passed to the limit. This allows us to derive observability properties for the limit system.
While this strategy is natural, we show that it can be used to establish new observability results in several settings, particularly in the presence of interfaces meeting the boundary, or multiple interfaces intersecting at a point,
under a suitable multiplier-type condition that imposes sign constraints on the jumps of the coefficients at the interfaces.
A key advantage of this approach is that it avoids the need for a detailed analysis of the regularity of solutions near the sets where the coefficients are singular.
\end{abstract}


\setcounter{tocdepth}{2}
\tableofcontents

%
% ------------------------------
%
\section{Introduction}
\label{Sec-Introduction}

\subsection{Main results}

In this article, we study observability inequalities for both the wave 
and plate equation in a bounded domain for a class of rough coefficients that includes the case of interfaces. 

For simplicity, we first consider the wave equation set in a bounded domain $\Omega \subset \R^d$ of class $\mathscr{C}^2$, given by:
\begin{equation}
	\label{Wave-eq-intro}
		\left\{
			\begin{array}{ll}
				\partial_{tt} y - \div (A \nabla y) = 0 & \text{ in } (0,T) \times \Omega, 
				\\
				y = 0 & \text{ on } (0,T) \times \partial \Omega, 
				\\ 
				(y, \partial_t y)\Big|_{t= 0} = (y_0, y_1) & \text{ in } \Omega,
			\end{array}
		\right.
\end{equation}
in which the main coefficient  $A \in L^\infty(\Omega; \R^{d\times d})$ is assumed to satisfy the following assumptions:
\begin{equation}
	\label{Ass-Symmetric-A}
		\text{$A$ is a symmetric matrix a.e. in $\Omega$ and } \exists \alpha>0, \text{s.t. } 
		\frac{1}{\alpha} I_{d\times d} \leq A(x) \leq \alpha I_{d\times d}, \quad \text{a.e.  } x\in\Omega. 
\end{equation}
(Here and in what follows, $I_{d\times d}$ denotes the identity matrix of size $d \times d$.)
In \eqref{Wave-eq-intro}, $y$ denotes the state, $(y_0, y_1)$ denotes the initial data, and $\div$ and $\nabla$ are, respectively,  the divergence and  gradient operators in $x$, as is customary.

Before going further, let us recall the following well-posedness result for system \eqref{Wave-eq-intro} (see \cite[Chapter III, Sections 8 and 9]{LionsMagenes}).
%% Thorem for WPP and def of E(t)
\begin{theorem}
	\label{Thm-WP}
	Let $T>0$ and assume that $A$ satisfies \eqref{Ass-Symmetric-A}. 
	
	Then for all $(y_0,y_1) \in H_0^1(\Omega)\times L^2(\Omega)$, there exists a unique  solution $y$
of system \eqref{Wave-eq-intro} in $\mathscr{C}^0([0,T]; H^1_0(\Omega)) \cap \mathscr{C}^1([0,T]; L^2(\Omega))$. Furthermore this solution satisfies that the energy, defined for $t \in [0,T]$ by  
	\begin{equation*}
		E(t):= \int_\Omega \left(   A \nabla y(t, \cdot) \cdot \nabla y(t, \cdot)  +  |\partial_t y(t, \cdot)|^2 \right) \ud x,
	\end{equation*}
	satisfies%
	\begin{equation*}
		\forall t \in [0,T], \quad E(t) = E(0) = \int_\Omega \left( A \nabla y_0(\cdot) \cdot \nabla y_0(\cdot) +  |y_1(\cdot)|^2  \right) \ud x.
	\end{equation*}
\end{theorem}

Our primary objective is to establish observability estimates for the wave equation \eqref{Wave-eq-intro}. To be more precise, we consider $\Gamma_0$ an open subset of $\partial \Omega$ and a time horizon $T>0$, and we investigate whether there exists a constant $C_A>0$ such that for all $(y_0, y_1) \in H^1_0(\Omega) \times L^2(\Omega)$, the solution $y$ of \eqref{Wave-eq-intro} satisfies
\begin{equation}
	\label{Observability}
	E(0) \leq C_{A}^2 \int_0^T \int_{\Gamma_0} |\partial_\nu y |^2 \, \ud\sigma\, \ud t.
\end{equation}
(Here and below, $\nu$ denotes the unitary exterior normal to the boundary $\partial \Omega$, and $\partial_\nu y$ denotes the normal derivative of $y$ at the boundary $\partial \Omega$.)

Note that the right hand side in \eqref{Observability} is well-defined for initial data in the energy space $H^1_0(\Omega) \times L^2(\Omega)$ when the coefficient $A$ is $W^{1,\infty}$ in a neighborhood of $\partial \Omega$, see Corollary\ref{Cor-Conv-Normal-Trace} and Definition \ref{Def-normal-derivative-y} (we also refer to \cite[Theorem 4.1]{LasieckaLionsTriggiani} for a similar statement, although the proof should slightly be adapted to deal with coefficients $A$ which belong to $W^{1, \infty}$ only in a neighborhood of $\partial \Omega$): this property is sometimes called admissibility (see \cite[Section 7.1]{TWBook}) or hidden regularity property (see, e.g., \cite[Chapter 1, Section 4.1]{lionsHUM}).

Whether or not the observability estimate \eqref{Observability} holds for solutions of \eqref{Wave-eq-intro}, this actually depends on $\Omega$, $\Gamma_0$, $T>0$ and the coefficient $A$. For example, in the $1$-d case, for $\Omega = (0,1)$, $\Gamma_0 = \{0\}$ and 
$A = 1 $, by writing the solution of the wave equation using the characteristics,  one can easily be convinced that the  observability holds if and only if $T \geq 2$. 

The study of observability properties is motivated by their central role in stabilization and control theory (see, e.g., \cite{lionsHUM}). In particular, the observability inequality \eqref{Observability} is equivalent to the controllability of the wave equation with boundary control. While we focus on observability for clarity, these results can be directly applied to derive control properties. Observability is also a key tool in inverse problems, where one seeks to recover information about the solution of a wave equation from boundary measurements (see, e.g., \cite{BaudouinDeBuhanErvedoza} and the references therein).

There are two main approaches to addressing observability for the wave equation.

The first, popularized by Lions \cite{lionsHUM} and Komornik \cite{KomornikMultipliers}, and originating from the seminal article \cite{LopFatHo}, is the multiplier method, 
which yields a sufficient condition for observability to hold,
 also called the $\Gamma$-condition or exit-condition. 
For $A = I_{d\times d}$, the condition requires the existence of $x_0 \in \R^d$ such that the set $\Gamma_0$ contains 
%the shadow from  of a light emanating from 
the shadow of $\Omega$ cast by a light source at $x_0$, i.e. 
\begin{equation*}
	\Gamma_0 \supset \{ x \in \partial \Omega, \, (x-x_0) \cdot \nu > 0 \},
\end{equation*}
and $T > 2 \sup_{x \in \Omega} \{| x -x_0| \}$. The article \cite{Komornik-1989} shows that such technique can be adapted to coefficients $A \in \mathscr{C}^1$ provided a suitable multiplier type condition is assumed. One can even go further and develop Carleman estimates under suitable multiplier type conditions, see for instance \cite{Fu-Yong-Zhang-2007, Liu-2013}, which in particular imply the observability property \eqref{Observability}. The remark that it can be used even for a density $\rho \in \mathscr{C}^0$ satisfying the  appropriate multiplier condition has been done only recently in the work \cite{Dehman-Erv-2016}, see the discussion later.

The second approach is based on microlocal analysis and the propagation of singularities: following  the work  \cite{Rauch-Taylor-1974} by Rauch and Taylor, Bardos, Lebeau and Rauch derived in \cite{BLR-89,BLR-92} sufficient conditions for the observability \eqref{Observability} of waves for $A \in \mathscr{C}^\infty$, which were  proven to be necessary and sufficient in \cite{BurqGerard}. These necessary and sufficient conditions can be roughly stated as follows: all rays of Geometric Optics (with respect to the metric $A$) starting in $\overline\Omega$ at $t=0$ should meet the observation set $\Gamma_0$ at  a non-diffractive point before time $T$. 
This has later been adapted to deal with a metric $A \in \mathscr{C}^2$ (and $\Omega$ of class $\mathscr{C}^3$) in \cite{Burq-1997}. Only recently this has been improved to metric $A \in \mathscr{C}^1$ in \cite{Burq-Dehman-LeRousseau} in compact manifolds, and in domains $\Omega$ of class $\mathscr{C}^2$ in the works \cite{Burq-Dehman-LeRousseau-I, Burq-Dehman-LeRousseau-II}. Note that, to be precise, the rays of Geometric Optics are the space-time projection of the generalized bicharacteristics of Melrose-Sj\"ostrand, which solves an ODE in the space frequency domain whose coefficient involve $\nabla A$. In particular, for $A \in \mathscr{C}^2$, the generalized bicharacteristics rays are uniquely defined using the classical Cauchy-Lipschitz theorem, while no uniqueness of the generalized bicharacteristics rays is known in the case of a coefficient $A \in \mathscr{C}^1$ (only the existence of the generalized bicharacteristics is known using the Cauchy-Peano theorem).

The presence of interfaces corresponds to coefficients $A$ that are piecewise smooth, the interfaces being the discontinuities of $A$.  To our knowledge, all results so far dealing with interfaces require the interfaces to be closed hypersurfaces that do not meet one another nor the boundary and have some regularity, typically $\mathscr{C}^1$. This is for instance the case in the classical work  \cite[Chapter 6]{lionsHUM}, or in the recent work \cite{BIMO} (see also the references therein) under the condition that the interface is $\mathscr{C}^3$, both of which are based on some multiplier type arguments (note that \cite{BIMO} goes further by proving a Carleman type estimate). A microlocal approach was also developed in \cite{Gagnon-2023}, but under a similar condition that the interface does not meet the boundary and is $\mathscr{C}^3$.

We aim to derive observability results for coefficients $A \in L^\infty$. To achieve this, we employ a regularization process, constructing a sequence of regularized problems for which uniform observability estimates can be derived. We then pass these estimates to the limit.

The challenge lies in identifying conditions that guarantee uniform observability estimates. To address this, we impose weak multiplier-type conditions on the coefficients. While a similar approach was developed in \cite{Dehman-Erv-2016} for the wave equation  $\rho \partial_{tt} y - \Delta y = 0$ with $\rho \in \mathscr{C}^0$, extending it to the general case of principal coefficients in $L^\infty$ requires additional work, which is the focus of our analysis.

To properly state our main results, we assume the following conditions on the function $A$:
%%% Begin Hypothesis for a
\begin{Assumption}
	\label{Assumption1}
	There exists a smooth domain $\Omega_0$  containing $\overline\Omega$ such that the function  $A \in L^\infty(\Omega; \R^{d\times d})$ can be extended as a function $A \in L^\infty(\Omega_0; \R^{d\times d})$ (still denoted the same for simplicity) satisfying:
	\begin{itemize}
		\item $A$ is a symmetric matrix a.e. in $\Omega_0$ and  there exists $\alpha>0$ such that 
		\begin{equation}
			\label{Positivity-A}
			\frac{1}{\alpha} I_{d\times d} \leq A(x) \leq \alpha I_{d\times d}, \quad \text{a.e. in } \Omega_0. 
		\end{equation}
		
		\item There exists $ 0\leq \beta <2 $ such that 
		\begin{equation}
			\label{ObsCond-a}	
			\forall \xi \in \R^d, \quad
			((x\cdot \nabla A - \beta A ) \xi) \cdot \xi    \leq 0
			\text{ in the sense of } \mathscr{D}'(\Omega_0), 
		\end{equation} 
		in which, setting $A = (a_{i,j})_{i,j \in \{1, \cdots, d\}}$ the matrix $x \cdot \nabla A$ denotes the symmetric matrix $(x \cdot \nabla a_{i,j})_{i,j \in \{1, \cdots, d\}}$. 
		
		\item Further, we assume  that 
		\begin{equation}
			\label{Cond-on-0-a}
			\left\{
				\begin{array}{l}
					0 \notin \overline{\Omega_0}, 
					\\
					\hbox{or}
					\\
					\hbox{$A$ is Lipschitz in a neighborhood of $0$.}
				\end{array}
			\right.  
		\end{equation}
		
%		\item There exists a neighborhood $\mathcal{V}$ of $\Gamma$ in $\Omega_0$ such that $A \in W^{1, \infty}(\mathcal{V})$.
	\end{itemize}
\end{Assumption}
	%
%%%  END hypothesis for a

One of the main results of this work is the following:

%%  Main result. It was in section 3
\begin{theorem} \label{MainThmOBS}
    	Let $\Omega$ be a non-empty bounded open subset of $\R^d$ of class $\mathscr{C}^2$
        and let $A\in L^\infty(\Omega; \R^{d\times d})$ satisfy Assumption \eqref{Assumption1}.
        
	Set 
	\begin{equation}
		\label{Multiplier-Set-Gamma}
		\Gamma_0 = \{ x \in \partial \Omega, \, x \cdot \nu > 0  \},
	\end{equation}
	and assume that there exists a neighborhood $\mathcal{V}$ of $\Gamma_0$ in $\Omega_0$ such that $A \in W^{1, \infty}(\mathcal{V}; \R^{d\times d})$.
    %of Proposition \ref{Prop-Approx-a}. 

	Then for any $T$ satisfying 
	\begin{equation}
		\label{Cond-T-Wave}
		T > \frac{4 \sqrt{\alpha} \sup_{x \in \Omega} |x|}{2 -  \beta},
	\end{equation}
	there exists $C > 0$ such that  any solution $y$ of equation \eqref{Wave-eq-intro} with initial data $(y_0,y_1) \in H_0^1(\Omega) \times L^2(\Omega)$ satisfies 
	\begin{equation}\label{Obs-Boundary}
		\int_\Omega \left(  A \nabla y_0 \cdot \nabla y_0 + |y_1|^2 \right) \ud x 
	\leq 
		 C^2 \int_0^T\int_{\Gamma_0} \left| \partial_\nu y\right|^2 
	 	  (x \cdot \nu)\, \ud \sigma\, \ud t.  
	\end{equation}
\end{theorem}

Condition \eqref{ObsCond-a} is a weak multiplier type condition for the function $A$. Unlike  classical settings, we do not assume that $A$ is $\mathscr{C}^1$ nor that \eqref{ObsCond-a} holds pointwise. Instead, we assume that this latter condition holds in the sense of distributions, that is, for all $\varphi \in \mathscr{C}^\infty_c(\Omega_0; \R_+)$ , for all $\xi \in \R^d$, 
\begin{equation}
	\label{Eq-Formulation-Weak-Multiplier}
	 \int_{\Omega_0} \left(A \xi \cdot \xi\right) (\div ( x \varphi) + \beta \varphi) \, \ud x \geq 0. 
\end{equation}
This condition thus makes sense even for $A \in L^\infty(\Omega_0; \R^{d\times d})$, and allows to directly handle interfaces. 

For instance, let $\Omega_1$ be a non-empty open subset of $\Omega$ with Lipschitz boundary and $\overline\Omega_1 \subset \Omega$, and $\Omega_0$ be an open set which contains $\overline\Omega$. If $A$ is of the form 
 $$
 A(x) = \begin{cases}
     a_1(x) I_{d\times d}, \, \text{ for } x \in \Omega_1 \\
     a_2(x) I_{d\times d} , \, \text{ for }x \in  \Omega_0 \setminus \overline{\Omega_1},
 \end{cases}
 $$
 with $a_1 \in W^{1,1}(\Omega_1)$ and $a_2 \in W^{1,1}(\Omega_0 \setminus\overline\Omega_1)$, $A$ satisfies the multiplier condition \eqref{ObsCond-a} provided 
 $$
 	x \cdot \nabla a_1 \leq \beta a_1 \text{ in } \Omega_1, 
	\quad 
	x \cdot \nabla a_2 \leq \beta a_2 \text{ in } \Omega_0 \setminus \overline\Omega_1,
	\quad \text{ and } \quad
	(a_2 - a_1) x \cdot \nu_1 \leq 0\quad \text{a.e. on } \partial \Omega_1.
 $$
 where $\nu_1$ is the outer unit normal of $\partial \Omega_1$. This is obtained using the classical Green formula, which is valid in $\Omega_1$ and in $\Omega_0 \setminus \overline\Omega_1$ due to the fact that $\partial \Omega_1$ is assumed to be Lipschitz, see \cite[Chapter 3, Section 1, Theorem 1.1]{Necas-1967} (here, the normal is defined almost everywhere).
 
In particular, if $a_1$ and $a_2$ are positive constants and $a_2 < a_1$, we recover the condition $  x\cdot \vec \nu_1 \geq 0 $, corresponding to the property of $\Omega_1$ being star-shaped with respect to the origin, and we recover the setting considered in \cite[Chapter 6]{lionsHUM}. %and represented in the next figure.
%\begin{figure}[!h]
%	\begin{center}
%		\begin{tikzpicture}[scale=0.8]
%			% Usamos "smooth cycle" para que la curva sea cerrada y suave
%			\draw[thin] plot [smooth cycle, tension=0.7] 
%			coordinates {
%				(-3*0.5, -2*0.5) (-3.5*0.5, 1*0.5) (-1*0.5, 3*0.5) 
%				(2*0.5, 2.5*0.5) (3.5*0.5, 0.5*0.5) (2.5*0.5, -2.5*0.5) (0, -3*0.5)
%			};
%			% Etiqueta del dominio grande
%			%\node at (2, 2) {$\Omega_0$};
%			\node[text width=.25cm] at (0,-.55) {$a_1$};
%			\draw[thick, dotted, fill=white] plot [smooth cycle, samples=100, domain=0:360] 
%			(\x : {0.8 + 0.2*cos(6*\x)}); % '6' es el número de ondulaciones
%			%\node[text width=.25cm] at (0,-.45) {$a_1$};
%			\node[text width=.25cm] at (0,.25) {$a_1$};
%			\node[text width=.25cm] at (0,-1.15) {$a_2$};
%			%\draw (0,0) circle (2.2cm);
%			%\node[text width=.25cm] at (-1.6,0) {$\partial \Omega_1$};
%			\node[text width=.25cm] at (0,1.8) {$\tiny{\Omega}$};
%			%\node[text width=.25cm] at (-2.52,0) {$\partial \Omega$};
%			\draw[->](1,0) -- (1.4,0) node[midway,above] {{\tiny{$\vec{\nu}_1$}}};
%			%\draw[->](2,0) -- (2.6,0) node[midway,above] {$\nu$};
%		\end{tikzpicture}
%	\end{center}
%	\caption{Assumption \eqref{ObsCond-a}  holds if $a_1$ and $a_2$  are positive constants with $a_2 < a_1$ and $\Omega_1$(dotted) star shaped with respect to the origin.}\label{lions-star-shaped.}
%\end{figure}
When $a_1$ and $a_2$ are variable (and ordered), we recover conditions similar to those obtained in \cite{BIMO} for these coefficients; however, in this reference some additional assumptions are imposed on $\Omega_1$, namely that it has a $\mathscr{C}^3$ boundary and that it is convex.

%When $a_1$ and $a_2$ are variable (and ordered), we recover conditions similar to the ones in \cite{BIMO}, which further imposes that $\Omega_1$ has a $\mathscr{C}^3$ boundary and that it is convex.

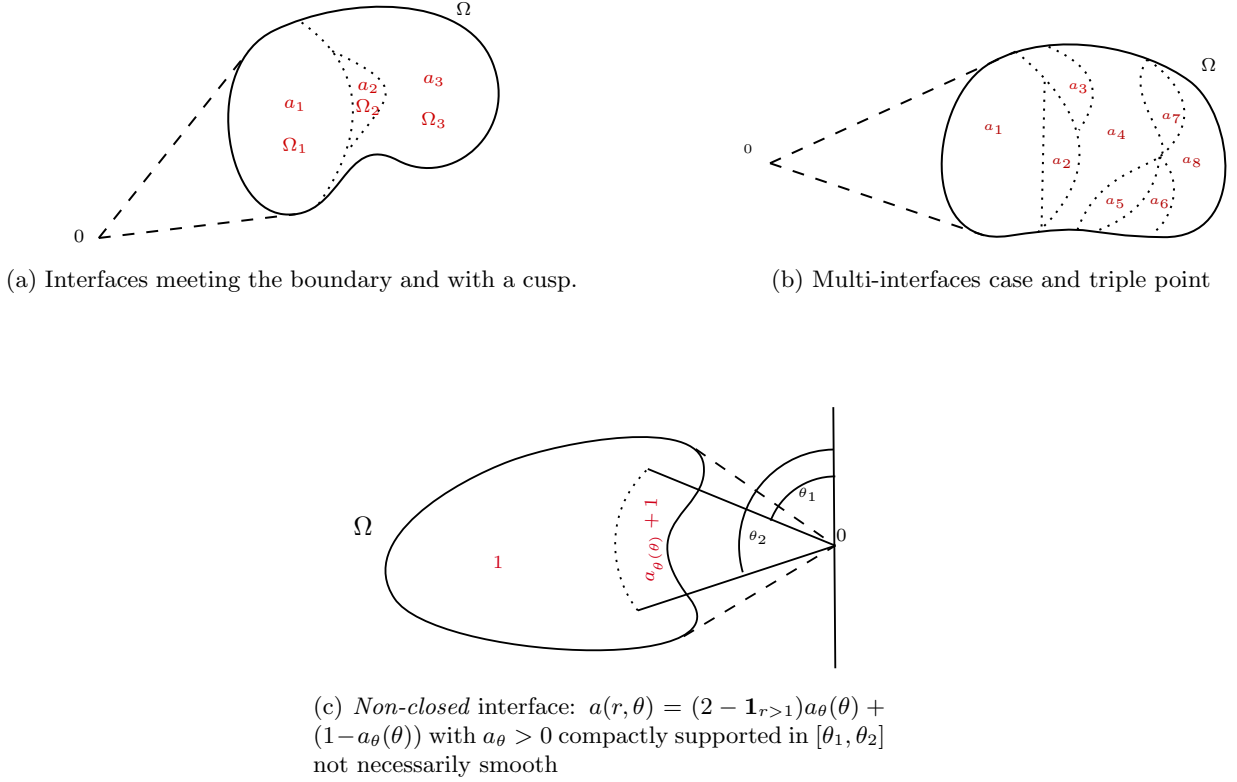
\begin{figure}[H]
	\centering
	% --- PRIMERA FILA: DOS IMÁGENES ---
	\begin{subfigure}[b]{0.45\textwidth}
		\centering
		\tikzset{every picture/.style={line width=0.75pt}} 
		\begin{tikzpicture}[x=0.75pt,y=0.75pt,yscale=-1,xscale=1, shift={(-430,-420)}]
			%Shape: Polygon Curved
			\draw (537.81,403.15) .. controls (551.58,437.33) and (517.99,467.14) .. (490.39,452.44) .. controls (462.79,437.73) and (462.93,482.65) .. (433.04,479.47) .. controls (403.14,476.29) and (391.73,404.83) .. (428.45,387.25) .. controls (465.17,369.66) and (524.04,368.96) .. (537.81,403.15) -- cycle ;
			\draw [dash pattern={on 0.84pt off 2.51pt}] (442.21,383.27) .. controls (484.55,413.26) and (465.35,454.03) .. (449.86,475.49) ;
			\draw [dash pattern={on 4.5pt off 4.5pt}] (340.5,491.39) -- (415.65,397.61) ;
			\draw [dash pattern={on 4.5pt off 4.5pt}] (340.5,491.39) -- (439.15,479.47) ;
			\draw [dash pattern={on 0.84pt off 2.51pt}] (457.51,397.58) .. controls (496.51,415.87) and (487.33,419.05) .. (465.92,446.87) ;
			\draw (438.93,424.81) node [font=\scriptsize, red!80!black, rotate=-358] {$a_{1}$};
            \draw (438.93,444.81) node [font=\scriptsize, red!80!black, rotate=-358] {$\Omega_{1}$};
			\draw (475.68,415.83) node [font=\scriptsize, red!80!black] {$a_{2}$};
			\draw (475.68,425.83) node [font=\scriptsize, red!80!black] {$\Omega_{2}$};
            \draw (514.68,368.31) node [anchor=north west, font=\scriptsize] {$\Omega $};
			\draw (323.1,482.39) node [anchor=north west, font=\scriptsize] {${0}$};
			\draw (508.57,412.42) node [font=\scriptsize, red!80!black] {$a_{3}$};
            \draw (508.57,432.42) node [font=\scriptsize, red!80!black] {$\Omega_{3}$};
		\end{tikzpicture}
		\caption{Interfaces meeting the boundary and with a cusp.}
	\end{subfigure}
	\hfill
	\begin{subfigure}[b]{0.45\textwidth}
		\centering
		\tikzset{every picture/.style={line width=0.75pt}} 
		\begin{tikzpicture}[x=0.75pt,y=0.75pt,yscale=-1,xscale=1, shift={(-160,-270)}]
			\draw (268.05,243.11) .. controls (290.4,259.86) and (297.4,321.86) .. (254.81,321.24) .. controls (212.22,320.62) and (223.65,313.56) .. (177.35,320.57) .. controls (131.05,327.57) and (136.65,251.1) .. (168.44,233.64) .. controls (200.23,216.18) and (245.71,226.36) .. (268.05,243.11) -- cycle ;
			\draw [dash pattern={on 0.84pt off 2.51pt}] (180.66,228.12) .. controls (217.31,257.88) and (225.26,288.88) .. (190.01,320.45) ;
			\draw [dash pattern={on 4.5pt off 4.5pt}] (57.52,284.15) -- (164.74,320.32) ;
			\draw [dash pattern={on 4.5pt off 4.5pt}] (57.52,284.15) -- (180.66,228.12) ;
			\draw [dash pattern={on 0.84pt off 2.51pt}] (197.21,225.75) .. controls (217.07,234.43) and (227,246.27) .. (212.44,267.58) ;
			\draw [dash pattern={on 0.84pt off 2.51pt}] (195.23,243.31) -- (193.65,316.56) ;
			\draw [dash pattern={on 0.84pt off 2.51pt}] (246.87,232.06) .. controls (278.64,254.95) and (261.43,273.89) .. (252.82,280.99) ;
			\draw [dash pattern={on 0.84pt off 2.51pt}] (252.82,280.99) .. controls (265.4,292.04) and (259.45,310.98) .. (250.84,318.08) ;
			\draw [dash pattern={on 0.84pt off 2.51pt}] (246.87,232.06) .. controls (233.65,249.56) and (261.43,273.89) .. (252.82,280.99) ;
			\draw [dash pattern={on 0.84pt off 2.51pt}] (222.65,317.56) .. controls (237.65,311.56) and (250.84,301.51) .. (252.82,280.99) ;
			\draw [dash pattern={on 0.84pt off 2.51pt}] (211.65,317.56) .. controls (220.92,291.52) and (243.56,287.31) .. (252.82,280.99) ;
			\draw (169.88,266.98) node [font=\tiny, red!70!black] {$a_{1}$};
			\draw (204,283.88) node [font=\tiny, red!70!black] {$a_{2}$};
			\draw (39.32,269.73) node [anchor=north west, font=\tiny] {${0}$};
			\draw (212.28,246.47) node [font=\tiny, red!80!black] {$a_{3}$};
			\draw (231.48,269.35) node [font=\tiny, red!70!black] {$a_{4}$};
			\draw (230.76,304.29) node [font=\tiny, red!70!black] {$a_{5}$};
			\draw (253,304.08) node [font=\tiny, red!70!black] {$a_{6}$};
			\draw (259,261.46) node [font=\tiny, red!70!black] {$a_{7}$};
			\draw (269.22,283.56) node [font=\tiny, red!70!black] {$a_{8}$};
			\draw (268.68,226.31) node [anchor=north west, font=\scriptsize] {$\Omega $};
		\end{tikzpicture}
		\caption{Multi-interfaces case and triple point}
	\end{subfigure}

	\vspace{0.5cm} % Espacio vertical entre filas

	% --- SEGUNDA FILA: TERCERA IMAGEN CENTRADA ---
	\begin{subfigure}[b]{0.45\textwidth}
		\centering
		% Aquí pones el código de tu tercera imagen TikZ

\tikzset{every picture/.style={line width=0.7pt}} %set default line width to 0.75pt        

\begin{tikzpicture}[x=0.75pt,y=0.75pt,yscale=-1,xscale=1]
%uncomment if require: \path (0,546); %set diagram left start at 0, and has height of 546

%Shape: Polygon Curved [id:ds35179030314133297] 
\draw   (157,401.75) .. controls (184,391.75) and (251,379.75) .. (254,403.75) .. controls (257,427.75) and (217,431.75) .. (247,470.75) .. controls (277,509.75) and (119,500.75) .. (99,470.75) .. controls (79,440.75) and (130,411.75) .. (157,401.75) -- cycle ;
%Curve Lines [id:da49208371498184134] 
\draw  [dash pattern={on 0.84pt off 2.51pt}]  (221,476.75) .. controls (208,460.75) and (202,428.75) .. (226,405.75) ;
%Straight Lines [id:da5843614019465139] 
\draw    (319,374.75) -- (320,505.75) ;
%Straight Lines [id:da37510062242619513] 
\draw    (226,405.75) -- (320,444.25) ;
%Straight Lines [id:da2762614068954379] 
\draw    (221,476.75) -- (320,444.25) ;
%Shape: Arc [id:dp9368384563184986] 
\draw  [draw opacity=0] (287.57,431.74) .. controls (292.67,418.49) and (305.53,409.48) .. (320,409.5) -- (320,444.25) -- cycle ; \draw   (287.57,431.74) .. controls (292.67,418.49) and (305.53,409.48) .. (320,409.5) ;  
%Shape: Arc [id:dp2182464384767021] 
\draw  [draw opacity=0] (273.61,457.58) .. controls (272.4,453.34) and (271.75,448.87) .. (271.75,444.25) .. controls (271.75,417.93) and (292.82,396.54) .. (319.01,396.01) -- (320,444.25) -- cycle ; \draw   (273.61,457.58) .. controls (272.4,453.34) and (271.75,448.87) .. (271.75,444.25) .. controls (271.75,417.93) and (292.82,396.54) .. (319.01,396.01) ;  
%Straight Lines [id:da9244941888241871] 
\draw  [dash pattern={on 4.5pt off 4.5pt}]  (250,395.75) -- (320,444.25) ;
%Straight Lines [id:da4042433010304728] 
\draw  [dash pattern={on 4.5pt off 4.5pt}]  (244,489.75) -- (320,444.25) ;

% Text Node
\draw (76.98,428.41) node [anchor=north west][inner sep=0.75pt]    {$\Omega $};
% Text Node
\draw (317.86,339.75) node [anchor=north west][inner sep=0.75pt]    {${}$};
% Text Node
\draw (319,434.55) node [anchor=north west][inner sep=0.75pt]  [font=\scriptsize]  {$0$};
% Text Node
\draw (300,413.55) node [anchor=north west][inner sep=0.75pt]  [font=\tiny]  {$\theta _{1}$};
% Text Node
\draw (275,435.55) node [anchor=north west][inner sep=0.75pt]  [font=\tiny]  {$\theta _{2}$};
% Text Node
\draw (150.97,453.23) node  [font=\scriptsize,color={rgb, 255:red, 208; green, 2; blue, 27 }  ,opacity=1 ,rotate=-0.52]  {$ \begin{array}{l}
1\\
\end{array}$};
% Text Node
\draw (228.97,441.23) node  [font=\scriptsize,color={rgb, 255:red, 208; green, 2; blue, 27 }  ,opacity=1 ,rotate=-272.39]  {$ 
a_{_{\theta }( \theta )} +1$};

\end{tikzpicture}
\caption{ {\it Non-closed} interface: $a(r,\theta)=(2-\mathbf{1}_{r>1})a_\theta(\theta) +(1-a_\theta(\theta))$ with $a_\theta>0$ compactly supported in $[\theta_1, \theta_2]$ not necessarily smooth}
	\end{subfigure}
\hspace{1cm}
	\caption{Some examples including interfaces (dotted line) where Assumption \ref{ObsCond-a} holds, provided the coefficients $a_i$ are suitably ordered in Figures (a)-(b) with $a_i \leq a_j$ for $i > j$.}\label{Figures}
\end{figure}

It is also clear from the above example that in fact interpreting condition \eqref{ObsCond-a} relies on the validity of the Gauss-Green theorem in domains delimited by the interfaces. As we said, the Gauss-Green theorem in domains is valid when the boundary of the domain is assumed to be Lipschitz, see \cite[Chapter 3, Section 1, Theorem 1.1]{Necas-1967}, but in fact it also makes sense for vector fields $f$ in $\mathscr{C}^1_c(\R^d; \R^d)$ for domains $E$ having finite perimeter, allowing to write
$$
	\int_E \div (f ) \, \ud x = \int_{\partial_* E}  f \cdot \nu_E \ud \mathcal{H}^{d-1}, 
$$
see \cite[Theorem 5.16]{Evans-Gariepy}, where $\partial_* E$ is the measure theoretic boundary of $E$ (see  \cite[Definition 5.7]{Evans-Gariepy}), and $\nu_E$ is the measure theoretic unit outer normal to $E$ (see  \cite[Definition 5.6]{Evans-Gariepy}). In particular, this allows to consider domains with cusps when the functions are $\mathscr{C}^1$.

It thus appears clearly that Assumption \ref{Assumption1} encompasses many more examples, including interfaces that meet the boundary or that may intersect one another, see Figure \ref{Figures} for some illustrations. Figures \ref{Figures}-(a) and \ref{Figures}-(b) present cases in which the coefficients $A$ are piecewise constants, and equal to $a_i$ in domains $\Omega_i$ delimited by the dotted lines. In such cases Assumption \ref{ObsCond-a} holds, provided the coefficients $a_i$ are suitably ordered in Figures (a)-(b) with $a_i \leq a_j$ for $i > j$. Finally, Figure \ref{Figures}-(c) presents a case in which the interface is non-closed. 

A major advantage of our approach is that it avoids the need for detailed regularity analysis of solutions near the singular sets of the coefficients. Instead, we approximate solutions of the wave equation \eqref{Wave-eq-intro} with rough coefficients by solutions of the same equation with smooth coefficients. For the latter, all computations can be justified using classical density arguments, and the results are then passed to the limit.

The cornerstone of this approach is Theorem \ref{Thm-Regularization}, which proves the strong convergence of solutions in the energy space $H^1_0(\Omega) \times L^2(\Omega)$ under weak convergence of the coefficients,  namely $L^1$ and almost everywhere convergence (see \eqref{Convergence-a-n-a}).

\bigskip

When the observation is done in an open set $\omega$ which is a neighborhood of the domain $\Gamma_0$ in \eqref{Multiplier-Set-Gamma}, we have the following counterpart to Theorem \ref{MainThmOBS}:
\begin{theorem} \label{MainThmOBS-distributed}
    	Let $\Omega$ be a non-empty bounded open subset of $\R^d$ of class $\mathscr{C}^2$
        and     
	let $A\in L^\infty(\Omega; \R^{d\times d})$ satisfy Assumption \eqref{Assumption1}.
        Assume that $\omega$ is an open subset of $\Omega_0$ satisfying
	\begin{equation}
		\label{Multiplier-Condition-omega}
		\omega \supset \{ x \in \partial \Omega, \, x \cdot \nu > 0 \}.
	\end{equation}

	Then for any $T$ satisfying \eqref{Cond-T-Wave}, 
	there exists $C > 0$ such that  any solution $y$ of equation \eqref{Wave-eq-intro} with initial data $(y_0,y_1) \in H_0^1(\Omega) \times L^2(\Omega)$ satisfies 
	\begin{equation}\label{Obs-Distributed}
		\int_\Omega \left( A \nabla y_0 \cdot \nabla y_0 + |y_1|^2 \right) \ud x 
	\leq 
		 C \int_0^T\int_{\omega \cap \Omega} \left| \partial_t y\right|^2 
	 	  \, \ud x\, \ud t.  
	\end{equation}
	
\end{theorem}

Our approach also applies for the plate equation in the presence of interfaces, and yields the following result, that we state only,  for the sake of simplicity, for operators of the form $\Delta( a \Delta \cdot)$ for a scalar function $a \in L^\infty(\Omega)$ with clamped boundary conditions:
\begin{theorem}\label{MainThmOBS-Plates}
    	Let $\Omega$ be a non-empty bounded open subset of $\R^d$ of class $\mathscr{C}^2$. Let $a \in L^\infty(\Omega)$ be such that it can be extended in a smooth domain $\Omega_0$ containing $\overline\Omega$, and the extension (still denoted the same) satisfies the following properties:
	\begin{itemize}
		\item There exists $\alpha >0$ such that $1/\alpha \leq a(x) \leq \alpha $ almost everywhere in $\Omega_0$.
		\item  There exists $0\leq \beta < 4$ such that 
		\begin{equation}
			\label{Multiplier-Condition-Plate}
			x \cdot \nabla a - \beta a \leq 0, \quad \text{ in the sense of $\mathscr{D}'(\Omega_0)$}. 
		\end{equation}
		\item $0 \notin \overline{\Omega_0}$ or $a$ is Lipschitz in a neighborhood of $0$. 
	\end{itemize}
	Let $\omega$ be an open subset of $\Omega_0$ satisfying \eqref{Multiplier-Condition-omega}. 
	
	Then, for any $T >0$, there exists $C>0$ such that any solution $y$ of 
	\begin{equation}
		\label{Plate-Eq}
		\left\{
			\begin{array}{ll}
				\partial_{tt} y + \Delta (a \Delta y) = 0 & \text{ in } (0,T) \times \Omega, 
				\\
				y = \partial_\nu y = 0 & \text{ on } (0,T) \times \partial \Omega, 
				\\ 
				(y, \partial_t y)\Big|_{t= 0} = (y_0, y_1) & \text{ in } \Omega,
			\end{array}
		\right.		
	\end{equation}
	with initial data $(y_0, y_1) \in H^2_0(\Omega) \times L^2(\Omega)$ satisfies
	\begin{equation}
		\label{Obs-Plate-Dis}
		\int_\Omega \left( a |\Delta y_0|^2 + |y_1|^2 \right) \, \ud x
		\leq 
		C \int_0^T \int_{\omega\cap \Omega} |\partial_t y|^2 \, \ud x\, \ud t. 
	\end{equation}

	Similarly, if $\Gamma_0$ is defined by  \eqref{Multiplier-Set-Gamma} and $a$ additionally is Lipschitz in a neighborhood of $\Gamma_0$, for any $T >0$, there exists $C>0$ such that any solution $y$ of \eqref{Plate-Eq}
	with initial data $(y_0, y_1) \in H^2_0(\Omega) \times L^2(\Omega)$ satisfies
	\begin{equation}
		\label{Obs-Plate-Gamma}
		\int_\Omega \left( a |\Delta y_0|^2 + |y_1|^2 \right) \, \ud x
		\leq 
		C \int_0^T \int_{\Gamma_0} |\Delta y|^2 \, \ud \sigma\, \ud t. 
	\end{equation}
\end{theorem}

Recall that, similarly as Theorem \ref{Thm-WP} for the wave equation, the plate equations \eqref{Plate-Eq} with coefficient $a \in L^\infty(\Omega)$ bounded from below by a positive constant is well-posed, see \cite[Chapter III, Sections 8 and 9]{LionsMagenes}). Solutions $y$ of \eqref{Plate-Eq} with initial data in $H^2_0(\Omega) \times L^2(\Omega)$ belong to $\mathscr{C}^0([0,T]; H^2_0(\Omega))\cap \mathscr{C}^1([0,T]; L^2(\Omega))$. Furthermore the energy of the solutions $y$ of \eqref{Plate-Eq}, defined for $t \in [0,T]$ by 
\begin{equation}
	\label{Energy-Plate}
    E(t) = \frac{1}{2} \int_{\Omega} \left( a(x)|\Delta y(t,x)|^2+ |\partial_t y(t,x)|^2  \right)  \, \ud x,
\end{equation}
is constant with respect to $t$, so that the left-hand side of \eqref{Obs-Plate-Dis} (or \eqref{Obs-Plate-Gamma}) coincides with the energy of the solutions.

Similarly to the wave equation, the multiplier condition \eqref{Multiplier-Condition-Plate} is not required to hold pointwise; instead, it is assumed in the sense of distributions. This formulation allows us to recover the case of transmission conditions for plates, originally studied in \cite{LiuWilliams} which,  like the wave counterparts, impose a sign condition on the jumps of the coefficients and require the inner domain to be star shaped. Consequently, our approach provides a natural extension for interface configurations like the ones in Figure \ref{Figures} also for the plate equation. 

The approach is similar to that we have developed for the wave equation: we carry out a regularization process and prove the strong convergence of the regular solutions and  uniform observability inequalities, which are obtained by employing multiplier techniques, this time particularly adapted to handle clamped boundary conditions.
%% of the form 
%$x \cdot \nabla y + \lambda y$, 

For a sufficiently large time $T>0$, we obtain a direct observability estimate (either from the boundary $\Gamma_0$ or from a distributed subset which is a neighborhood of $\Gamma_0$) by the multiplier method. However, to extend this result to an arbitrary time $T>0$, we first establish a weak observability inequality that includes compact terms of the form $\left(\|\nabla y_0\|_{L^2(\Omega)}^2 + \|\nabla y(T)\|_{L^2(\Omega)}^2\right)$ on the right-hand side. These lower-order terms are subsequently removed by means of a compactness-uniqueness argument.

\subsection{Additional comments}

\paragraph{Carleman estimates.} Given the above results, it is natural to ask whether Carleman estimates hold for the wave and plate equations with rough coefficients under the same weak multiplier conditions. However, this remains an open problem, as existing proofs of Carleman estimates typically require coefficients to be $\mathscr{C}^1$ or piecewise $\mathscr{C}^1$ with smooth interfaces, see \cite{Fu-Yong-Zhang-2007, Liu-2013,BIMO} and references therein.
In a forthcoming work, we will show that Carleman estimates for the $1$-d wave equation hold for coefficients in the BV (bounded variation) class.

\paragraph{Consequences on observability and control of heat equation with rough coefficients.}

The so-called transmutation method \cite{Miller-2006} (see also \cite{Ervedoza-Zuazua}) shows that the exact observability of the wave equation in some time implies the observability at any final time for the corresponding heat equation. Accordingly, our results imply the observability at any final time of heat operators of the form $\partial_t - \div(A \nabla \cdot)$ (with Dirichlet boundary conditions) for any $A$ satisfying Assumption \ref{Assumption1}, allowing in particular rough coefficients and interfaces, providing a suitable weak multiplier condition is satisfied. 

While the weak multiplier condition in Assumption \ref{Assumption1} is very likely not needed for the final time observability of the heat equation (when the coefficient is piecewise $\mathscr{C}^1$ and the discontinuities form closed interfaces, it reminds the geometric limitations in the work \cite{Doubova-Osses-Puel}, which were later shown to be unnecessary, see \cite{LeRousseau-Robbiano-2010, LeRousseau-Robbiano-2011,LeRousseau-Lerner-2013}), this method allows to deal with geometric cases such as the ones presented in Figure \ref{Figures}, which to our knowledge cannot be handled in the literature so far. 

\paragraph{Weak structural conditions.} Our approach demonstrates that appropriate structural conditions can help overcome typical regularity thresholds for coefficients or unwanted behaviors. This is relevant in several contexts:

{\bf Numerical computation of exact controls.} 
Since the pioneering work of \cite{Glowinski-Li-Lions}, it is known that discretizing the wave equation and computing the exact control for the discrete system fails to provide good approximations of the continuous control due to spurious high-frequency solutions \cite{Zuazua-2005-SIAM}. This occurs because discrete wave equations are generally not uniformly observable, even when the continuous system is observable.

One way to restore uniform observability is to design adapted numerical meshes. So far, this strategy has only been developed in $1$-d \cite{Ervedoza-Marica-Zuazua} by choosing meshes with a structure that aligns with a discrete multiplier argument. Extending this approach to higher dimensions remains an open problem.

{\bf Unique continuation for elliptic operators.}
We also mention the study of unique continuation properties for elliptic operators of the form $\div(A \nabla \cdot)$: if $u$ is a solution of $\div(A \nabla u ) = 0$ in $\Omega$ which vanishes in some open subset $\omega \subset \Omega$, is it true that $u$ vanishes everywhere? 
This property is true when $A$ belongs to $L^\infty$ when $d = 2$ (\cite{Bers-Niremberg-1955,Schulz-1998}), and when $A$ is Lipschitz when $d \geq 3$ (\cite{Wolff-1992,Koch-Tataru-2001}). In the other direction, there are counterexamples to unique continuation in dimension $d \geq 3$ due to \cite{Mandache-1998} (see also \cite{Miller-1974}), which are H\"older continuous to any order $\alpha \in (0,1)$. The last result we are aware of on this topic is the work \cite{Jeznach-2025}, with refined estimates on the modulus of continuity needed to guarantee unique continuation (allowing to consider log-Lipschitz classes), and we refer to it for further references. 
It would be interesting to investigate if unique continuation properties from a set $\omega$ containing $0$ can be proved for solutions of $\div(A \nabla u ) = 0$ in $\Omega$ vanishing in $\omega$ in cases in which $A$ satisfies multiplier type conditions of the form $x \cdot \nabla A \geq 0$ in $\mathscr{D}'$.

\bigskip

\noindent{\bf Outline.} The remainder of the paper is organized as follows: Section \ref{Sec-Regularization} is devoted to proving the convergence of the solutions of the wave equation in the corresponding energy space, under the assumption of convergence of the main coefficient. In Section \ref{MultSect}, observability estimates are established for solutions corresponding to regular main coefficients; moreover, using the convergence of these solutions, observability estimates are derived for solutions corresponding to an $L^\infty$ main coefficient. Theorems \ref{MainThmOBS}--\ref{MainThmOBS-distributed} and \ref{MainThmOBS-Plates} are respectively proven in Sections \ref{Subsec-Multiplier-wave} and \ref{Subsec-Obs-Plates}.
%Finally, Section \ref{Conclusion} collects some additional remarks.
%Finally, Section \ref{Carleman-Section} is devoted to obtain a Carleman-type inequality and its application to obtain the stability of an coefficient inverse problem. 

% ------------------------------
%
\section{The regularization process}
\label{Sec-Regularization}

The goal of this section is to prove the following result:
\begin{theorem}
	\label{Thm-Regularization}
	Let $\Omega$ be a smooth bounded domain of $\R^d$, $d \geq 1$, and let $T>0$.
	
	Let $A \in L^\infty(\Omega; \R^{d \times d})$ and a sequence $(A_n)_{n \in \N} \in L^\infty(\Omega; \R^{d\times d} )$ taking value in the set of symmetric matrices such that, for some $\alpha >0$.
	\begin{align}
		\label{Ass-WP-a}
		&
		\text{a.e. } x \in \Omega, \quad \frac{1}{\alpha} I_{d\times d} \leq A(x) \leq \alpha I_{d\times d}, 
		\\
		\label{Ass-WP-a-n}
		&
		\forall n \in \N,\, \text{a.e. } x \in \Omega, \quad \frac{1}{\alpha} I_{d\times d} \leq A_n(x) \leq \alpha I_{d\times d}, 
		\\
		\label{Convergence-a-n-a}
		& (A_n) \underset{n \to \infty} \longrightarrow A \text{ in } L^1(\Omega ; \R^{d \times d}), 
		\hbox{ and } 
		A_n \underset{n \to \infty} \longrightarrow A \text{ almost everywhere in } \Omega. 
	\end{align}
	Then, for any $(y_0, y_1) \in H^1_0(\Omega) \times L^2(\Omega)$, denoting by $y$ and $y_n$ the solutions of 
	\begin{equation}
		\label{Wave-eq-a}
		\left\{
			\begin{array}{ll}
				\partial_{tt} y - \div (A \nabla y) = 0 & \text{ in } (0,T) \times \Omega, 
				\\
				y = 0 & \text{ on } (0,T) \times \partial \Omega, 
				\\ 
				(y, \partial_t y)\Big|_{t= 0} = (y_0, y_1) & \text{ in } \Omega,
			\end{array}
		\right.
	\end{equation}
	and
	\begin{equation}
		\label{Wave-eq-a-n}
		\left\{
			\begin{array}{ll}
				\partial_{tt} y_n - \div (A_n \nabla y_n) = 0 & \text{ in } (0,T) \times \Omega, 
				\\
				y_n = 0 & \text{ on } (0,T) \times \partial \Omega, 
				\\ 
				(y_n, \partial_t y_n)\Big|_{t= 0} = (y_0, y_1) & \text{ in } \Omega,
			\end{array}
		\right.
	\end{equation}
	the sequence $(y_n)_{n \in \N}$ converges strongly to $y$ in $L^2(0,T; H^1_0(\Omega)) \cap H^1(0,T; L^2(\Omega))$ as $n \to \infty$, and the sequence $(y_n(T), \partial_t y_n(T))_{n \in \N}$ converges strongly to $(y(T), \partial_t y(T))$ in $H^1_0(\Omega)\times L^2(\Omega)$.
\end{theorem}

\begin{remark}
	\label{Rem-WP}
	The assumptions \eqref{Ass-WP-a}--\eqref{Ass-WP-a-n} ensure the well-posedness of the systems \eqref{Wave-eq-a} and \eqref{Wave-eq-a-n} in $H^1_0(\Omega) \times L^2(\Omega)$, and simply corresponds to the usual $L^\infty$ and ellipticity conditions on the coefficient $A$, respectively $A_n$.
\end{remark}

\begin{proof}
	%
%	First note that since the sequence $(A_n)_{n \in \N}$ is bounded in $L^\infty(\Omega)$ and strongly convergent to $A$ in $L^1(\Omega)$ as $n\to \infty$, it is weakly-$\star$ convergent in $L^\infty(\Omega)$ to $A$ as $n \to \infty$. 
	%
	\medskip
	
	{\it A priori estimates.} The energy identities immediately give that for all $t \in [0,T]$,
	\begin{align}
		\label{Energy-a}
		&
		\int_\Omega \left(  A \nabla y(t,x) \cdot \nabla y(t,x) +  |\partial_t y(t,x)|^2\right) \, \ud x 
		= 
		\int_\Omega \left(   A \nabla y_0(x) \cdot \nabla y_0(x)+ |y_1(x)|^2 \right) \, \ud x,  
		\\
		\label{Energy-a-n}
		&
		\forall n \in \N,
		\quad 
		\int_\Omega \left(   A_n \nabla y_n(t,x) \cdot \nabla y_n(t,x)+ |\partial_t y_n(t,x)|^2 \right) \, \ud x 
		= 
		\int_\Omega \left(  A_n \nabla y_0(x) \cdot \nabla y_0(x) + |y_1(x)|^2  \right) \, \ud x . 
	\end{align}
	Accordingly, since the sequence $(A_n)$ is uniformly bounded from below and from above (recall \eqref{Ass-WP-a}--\eqref{Ass-WP-a-n}, the sequence $(y_n)_{n \in \N}$ is uniformly bounded in $\mathscr{C}^0([0,T]; H^1_0(\Omega)) \cap \mathscr{C}^1([0,T]; L^2(\Omega))$. Hence, up to a subsequence still denoted the same for simplicity, there exists $\tilde y  \in L^\infty(0,T; H^1_0(\Omega)) \cap W^{1, \infty}(0,T; L^2(\Omega))$ such that the sequence $(y_n)_{n \in \N}$ weakly converges to $\tilde y$ in $L^2(0,T; H^1_0(\Omega)) \cap H^1(0,T; L^2(\Omega))$. 
	\medskip
	
	{\it Identifying the weak limit $\tilde y$.} We define 
	$$
		\forall n \in \N, \quad z_n = y_n - y, \text{ in } (0,T) \times \Omega.
	$$
	Then $z_n$ satisfies the equation
	\begin{equation}
		\label{Wave-eq-z-n}
		\left\{
			\begin{array}{ll}
				\partial_{tt} z_n - \div (A_n \nabla z_n) = \div ((A_n - A) \nabla y ) & \text{ in } (0,T) \times \Omega, 
				\\
				z_n = 0 & \text{ on } (0,T) \times \partial \Omega, 
				\\ 
				(z_n, \partial_t z_n)\Big|_{t= 0} = (0, 0) & \text{ in } \Omega.
			\end{array}
		\right.
	\end{equation}
	To estimate $z_n$, we introduce the function $w_n$ given for all $t \in [0,T]$ by 
	\begin{equation}
		\label{Def-w-n}
		\left\{
			\begin{array}{ll}
				- \div (A_n \nabla w_n(t)) = z_n (t) & \text{ in } \Omega, 
				\\
				w_n(t) = 0 & \text{ on }  \partial \Omega. 
			\end{array}
		\right.
	\end{equation}
	Note that $z_n \in L^2(0,T; H^1_0(\Omega)) \cap H^1(0,T; L^2(\Omega))$, hence the classical elliptic regularity yields that $\partial_t w_n \in L^2(0,T; H^2(\Omega) \cap H^1_0(\Omega))$. We can thus  multiply \eqref{Wave-eq-z-n} by $\partial_t w_n$, and we get, for all $t \in [0,T]$:
	\begin{multline}
		\label{H--1-Energy-of-z}
		\int_\Omega \partial_{tt} \left( - \div (A_n \nabla w_n(t) )\right)  \partial_t w_n(t) \, \ud x 
		- 
		\int_\Omega \div (A_n \nabla z_n(t) ) \partial_t w_n(t) \, \ud x 
		\\
		= 
		- \int_\Omega  (A_n - A) \nabla y(t)  \cdot \nabla \partial_t w_n(t) \, \ud x.
	\end{multline}
	The first term gives: 
	\begin{align*}
		\int_\Omega \partial_{tt} \left( - \div (A_n \nabla w_n(t) )\right) \partial_t w_n(t) \, \ud x 
		&
		= 
		\int_\Omega A_n \partial_{tt} \nabla w_n(t) \cdot \partial_t \nabla  w_n(t) \, \ud x
		\\
		&= 
		\frac{d}{dt} \left(\frac{1}{2} \int_\Omega A_n \nabla \partial_t w_n \cdot \nabla \partial_t w_n \, \ud x \right).
	\end{align*}
	The second term yields: 
	\begin{align*}
		- \int_\Omega \div (A_n \nabla z_n(t) ) \partial_t w_n(t) \, \ud x 
		&
		= 
		\int_\Omega  z_n(t)  \partial_t \left( - \div (A_n \nabla w_n(t))\right) \, \ud x 
		\\ 
		&
		= 
		\frac{d}{dt} \left(\frac{1}{2} \int_\Omega |z_n(t)|^2 \, \ud x \right).
	\end{align*}
	For the last term, we simply write 
	\begin{align*}
		\left| 
			\int_\Omega  (A_n - A) \nabla y(t)  \cdot \nabla \partial_t w_n(t) \, \ud x
		\right|
		& 
		\leq 
		 \int_\Omega |(A_n -A) \nabla y(t)| | \nabla \partial_t w_n(t)| \, \ud x
		 \\
		 & 
		 \leq
		 \| (A_n -A) \nabla y(t) \|_{L^2(\Omega)} \| \nabla \partial_t w_n(t) \|_{L^2(\Omega)}.
	\end{align*}
	Together with \eqref{Ass-WP-a-n}, the equation \eqref{H--1-Energy-of-z} then gives that 
	$$
		\frac{d}{dt} \left( E_{w_n} (t) \right) \leq \frac{1}{\sqrt{\alpha}} \sqrt{2 E_w(t)} \| (A_n -A) \nabla y(t) \|_{L^2(\Omega)}, 
	$$
	where we have introduced the energy 
	$$
		E_{w_n}(t) = \frac{1}{2} \int_\Omega \left(   |\div (A_n \nabla w_n(t) )|^2 + A_n \nabla \partial_t w_n(t) \cdot \nabla \partial_t w_n(t)\right) \, \ud x, 
	$$
	which in fact contains the $L^{2}$-norm of the solutions $z_n$ of \eqref{Wave-eq-z-n}, as for all $t \in [0,T]$, 
	$$
		\int_\Omega |\div (A_n \nabla w_n(t))|^2 \, \ud x = \int_\Omega |z_n(t) |^2 \, \ud x
	$$
	from the definition \eqref{Def-w-n} of $w_n$.
	We then easily get that, for all $T>0$, there exists a constant $C>0$ such that for all $n \in \N$, 
	\begin{equation}
		\label{Estimates-on-w-n}
		\sup_{t \in [0,T]} \{ E_{w_n}(t) \} \leq C \| (A_n - A) \nabla y \|_{L^2 ((0,T) \times \Omega)}^2.
	\end{equation}
	Now, since $ |A_n - A|^2 |\nabla y|^2 \leq 4 \alpha^2 |\nabla y|^2$ and $\nabla y \in L^2((0,T) \times \Omega)$, and the sequence $(A_n)_{n \in \N}$ converges to $A$ almost everywhere in $\Omega$ as $n \to \infty$, the Lebesgue dominated convergence theorem implies that 
	$$
		\lim_{n \to \infty}  \| (A_n - A) \nabla y \|_{L^2 ((0,T) \times \Omega)}^2 = 0.
	$$
	Since
	$$
		\sup_{t \in [0,T]} \| z_n (t )\|_{L^2(\Omega)}^2
		\leq 
		2 \sup_{t \in [0,T]} \{ E_{w_n}(t) \} , 
	$$
	it follows that 
	$$
		\lim_{n \to \infty} \sup_{t \in [0,T]}\| z_n (t )\|_{L^2(\Omega)}^2 = 0.
	$$
	Therefore, the sequence $(y_n)_{n \in \N}$ converges to $y$ strongly in $L^2((0,T) \times \Omega)$ as $n \to \infty$, and thus the weak limit $\tilde y$ coincides with the solution $y$ of \eqref{Wave-eq-a}. 
	
	For later use, let us also point out that, as $(z_n(T))_{n \in \N}$ converges to $0$ in $L^2(\Omega)$ from the previous computations, the sequence $(y_n(T))_{n \in \N}$ also strongly converges to $y(T)$ in $L^2(\Omega)$. Similarly, we claim that the sequence $(\partial_t y_n(T))_{n \in \N}$ strongly converges to $\partial_t y(T)$ in $H^{-1}(\Omega)$. Indeed, using the fact that the Laplace operator with Dirichlet boundary conditions $-\Delta_D$ is an isomorphism between $H^1_0(\Omega)$ and $H^{-1}(\Omega)$ and the bounds in \eqref{Ass-WP-a}--\eqref{Ass-WP-a-n}, we have that 
	\begin{align*}
		\| \partial_t z_n (T) \|_{H^{-1}(\Omega)}^2 
		&
		= 
		\int_\Omega |\nabla ( (-\Delta_D)^{-1} \partial_t z_n(T) )|^2 \, dx
		\\
		&
		= 
		\int_\Omega ( (-\Delta_D)^{-1} \partial_t z_n(T) ) \partial_t z_n(T) \, dx
		\\ 
		&
		= 
		\int_\Omega ( (-\Delta_D)^{-1} \partial_t z_n(T) ) (- \div (A_n \nabla \partial_t w_n(T)) \, dx
		\\
		&
		= 
		\int_\Omega A_n \nabla ( (-\Delta_D)^{-1} \partial_t z_n(T) ) \cdot \nabla \partial_t w_n(T) \, dx
		\\
		& 
		\leq 
		\alpha 
		\| \nabla ( (-\Delta_D)^{-1} \partial_t z_n(T) )\|_{L^2(\Omega)}
		\| \nabla \partial_t w_n(T)\|_{L^2(\Omega)} 
		\\
		& 
		\leq 
		\alpha^2 \| \partial_t z_n(T)\|_{H^{-1}(\Omega)} 
		\left( \int_\Omega A_n \nabla \partial_t w_n(T) \cdot \nabla \partial_t w_n(T) \, dx \right)^{1/2}, 
	\end{align*}
	so that the convergence in \eqref{Estimates-on-w-n} yields that the sequence $(\partial_t z_n(T))_{n \in \N}$ strongly converges to $0$ in $H^{-1}(\Omega)$, that is,  the sequence $(\partial_t y_n(T))_{n \in \N}$ strongly converges to $\partial_t y(T)$ in $H^{-1}(\Omega)$. In particular, since we also know from \eqref{Energy-a-n} that the sequence $(\partial_t y_n(T))_{n \in \N}$ is bounded in $L^2(\Omega)$, we immediately have that the sequence $(\partial_t y_n(T))$ weakly converges to $\partial_t y(T)$ in $L^2(\Omega)$ as $n \to \infty$. 
	\medskip
	
	{\it Strong convergence in $L^2(0,T; H^1_0(\Omega)) \cap H^1(0,T; L^2(\Omega))$.} Since we already have that the sequence $(y_n)_{n \in \N}$ converges to $y$ weakly in $L^2(0,T; H^1_0(\Omega)) \cap H^1(0,T; L^2(\Omega))$ as $n \to \infty$, we only have to check that the sequence of the $L^2(0,T; H^1_0(\Omega)) \cap H^1(0,T; L^2(\Omega))$-norms of $y_n$ converges to the $L^2(0,T; H^1_0(\Omega)) \cap H^1(0,T; L^2(\Omega))$-norm of $y$.
	
	In order to do so, the energy identities \eqref{Energy-a} and \eqref{Energy-a-n} will strongly help us. Indeed, let us first remark that since $\nabla y_0 \in L^2(\Omega)$ and the sequence $A_n$ is weakly-$\star$ convergent to $A$ in $L^\infty(\Omega)$, 
	$$
		\lim_{n \to \infty} \left( \int_\Omega A_n \nabla y_0 \cdot \nabla y_0\, \ud x \right)
		=
		\int_\Omega A \nabla y_0 \cdot \nabla y_0\, \ud x.
	$$
	It then follows that from the energy identities \eqref{Energy-a} and \eqref{Energy-a-n}
	\begin{multline}
		\label{Convergence-of-the-energy}
		\lim_{n \to \infty} \sup_{t \in [0,T]} 
			\left( 
				\int_\Omega \left( |\partial_t y_n(t,x)|^2 + A_n \nabla y_n(t,x) \cdot \nabla y_n(t,x) \right) \, \ud x 
				\right.
				\\
				\left. 
				- 
				\int_\Omega \left( |\partial_t y(t,x)|^2 + A \nabla y(t,x) \cdot \nabla y(t,x) \right) \, \ud x
			\right) 
		= 0.
	\end{multline}

	We then show that the sequence $(\partial_t y_n)_{n \in \N}$ strongly converges in $L^2((0,T) \times \Omega)$ to $\partial_t y$. Here, the main ingredients are the following identities (sometimes called equipartition of the energy) and obtained after multiplying \eqref{Wave-eq-a} by $y$, respectively \eqref{Wave-eq-a-n} by $y_n$:
	\begin{align}
		\label{Equipartition-Energy-y}
		&
		\int_\Omega \partial_t y(t, x) y(t,x) \, \ud x \Big|_0^T 
		- 
		\int_0^T \int_\Omega |\partial_t y(t, x)|^2\, \ud x \ud t 
		+ 
		\int_0^T \int_\Omega A(x) \nabla y(t, x) \cdot \nabla y(t, x) \, \ud x \ud t 
		= 0,
		\\
		&
		\int_\Omega \partial_t y_n(t, x) y_n(t,x) \, \ud x \Big|_0^T 
		- 
		\int_0^T \int_\Omega |\partial_t y_n(t, x)|^2\, \ud x \ud t 
		+ 
		\int_0^T \int_\Omega A_n(x) \nabla y_n(t, x) \cdot \nabla y_n(t, x) \, \ud x \ud t 
		= 0.
		\notag
	\end{align}
	Accordingly, 
	\begin{align*}
		&
		2 \int_0^T \int_\Omega |\partial_t y_n(t, x)|^2\, \ud x \ud t 
		- 
		2 \int_0^T \int_\Omega |\partial_t y(t, x)|^2\, \ud x \ud t 
		\\
		&
		 =  
		\int_\Omega \partial_t y_n(t, x) y_n(t,x) \, \ud x \Big|_0^T
		- 
		\int_\Omega \partial_t y(t, x) y(t,x) \, \ud x \Big|_0^T 
		\\
		& 
		+  
		\int_0^T \int_\Omega \left( |\partial_t y_n(t,x)|^2 + A_n \nabla y_n(t,x) \cdot \nabla y_n(t,x) \right) \, \ud x\ud t
		- 
		\int_0^T \int_\Omega \left( |\partial_t y(t,x)|^2 + A \nabla y(t,x) \cdot \nabla y(t,x) \right) \, \ud x\ud t
	\end{align*}
	Since $(y_n, \partial_t y_n)\Big|_{t= 0} =  (y, \partial_t y)\Big|_{t= 0}$, and the sequence $(\partial_t y_n(T))_{n \in \N}$ weakly converges to $\partial_t y(T)$ in $L^2(\Omega)$ and the sequence $(y_n(T))_{n \in \N}$ strongly converge to $y(T)$ in $L^2(\Omega)$, in view of the convergence \eqref{Convergence-of-the-energy}, we get that 
	$$
		\lim_{n \to \infty}  \int_0^T \int_\Omega |\partial_t y_n(t, x)|^2\, \ud x \ud t 
		=
		\int_0^T \int_\Omega |\partial_t y(t, x)|^2\, \ud x \ud t, 
	$$
	which of course entails the strong convergence of the sequence $(\partial_t y_n)_{n \in \N}$ to $\partial_t y$ as $n \to \infty$ in $L^2((0,T) \times \Omega)$. Note that, combined with \eqref{Convergence-of-the-energy}, we also get 
	\begin{equation}
		\label{Convergence-a-n-nabla-n}
		\lim_{n \to \infty} \int_0^T 
			\int_\Omega A_n(x) \nabla y_n(t,x) \cdot \nabla y_n(t,x) \, \ud x \ud t 
			=
			\int_0^T \int_\Omega A(x) \nabla y(t,x) \cdot \nabla y(t,x)  \, \ud x \ud t.
	\end{equation}
	
	Now, we focus on the convergence of the sequence $(y_n)_{n \in \N}$ in $L^2(0,T; H^1_0(\Omega))$. To start with, multiplying the equation \eqref{Wave-eq-a-n} by $y$, we get for all $n \in \N$, 
	$$
		\int_\Omega \partial_t y_n(t, x) y(t,x) \, \ud x \Big|_0^T 
		- 
		\int_0^T \int_\Omega \partial_t y_n(t, x) \partial_t y(t, x) \, \ud x \ud t 
		+ 
		\int_0^T \int_\Omega A_n(x) \nabla y_n(t, x) \cdot \nabla y(t, x) \, \ud x \ud t 
		= 0. 
	$$
	Since we can pass to the limit in the first two terms (as $(\partial_t y_n(T))_{n \in \N}$ weakly converges to $\partial_t y(T)$ in $L^2(\Omega)$ and $(\partial_t y_n)_{n \in \N}$ strongly converges to $\partial_t y$ in $L^2((0,T) \times \Omega)$), we have that 
	$$
		\lim_{n \to \infty} \int_0^T \int_\Omega A_n(x) \nabla y_n(t, x) \cdot \nabla y(t, x) \, \ud x \ud t 
		= 
		\int_0^T \int_\Omega |\partial_t y(t, x)|^2 \, \ud x \ud t 
		- 
		\int_\Omega \partial_t y(t, x) y(t,x) \, \ud x \Big|_0^T.
	$$
	Using then the equipartition of the energy \eqref{Equipartition-Energy-y} for $y$, we deduce that 
	\begin{equation}
		\label{Convergence-a-n-nabla-n-y}
		\lim_{n \to \infty} \int_0^T \int_\Omega A_n(x) \nabla y_n(t, x) \cdot \nabla y(t, x) \, \ud x \ud t 
		= 
		\int_0^T \int_\Omega A(x) \nabla y(t, x) \cdot \nabla y(t, x) \, \ud x \ud t.
	\end{equation}
	We then write that 
	\begin{align*}
		\frac{1}{\alpha} \| \nabla (y_n- y)\|_{L^2((0,T) \times \Omega)}^2
		&
		\leq
		\int_0^T \int_\Omega A_n(x) \nabla (y_n (t,x)- y(t,x) ) \cdot \nabla (y_n(t,x) - y(t,x)) \,  \ud x \ud t 
		\\
		&
		\hspace{-2cm} 
		\leq 
		\int_0^T \int_\Omega A_n(x) \nabla y_n (t,x) \cdot \nabla y_n (t,x) \,  \ud x \ud t 
		- 
		2\int_0^T \int_\Omega A_n(x) \nabla y_n(t, x) \cdot \nabla y(t, x) \, \ud x \ud t 
		\\
		&
		\hspace{-1.8cm} 
		+ 
		\int_0^T \int_\Omega A_n(x) \nabla y(t,x)  \cdot \nabla y(t,x)  \,  \ud x \ud t, 
	\end{align*}
	and, from \eqref{Convergence-a-n-nabla-n} and \eqref{Convergence-a-n-nabla-n-y}, we deduce that 
	$$
		\lim_{n \to \infty}  \| \nabla (y_n- y)\|_{L^2((0,T) \times \Omega)}^2 = 0, 
	$$
	i.e. that $y_n$ strongly converges to $y$ in $L^2(0,T; H^1_0(\Omega))$.
    \medskip
    
    {\it Strong convergence of $(y_n(T), \partial_t y_n(T))$ in $H^1_0(\Omega)\times L^2(\Omega)$.} Note that we have already proved that the sequence $(y_n(T), \partial_t y_n(T))_{n \in \N}$ weakly converges to $(y(T), \partial_t y(T))$ in $H^1_0(\Omega)\times L^2(\Omega)$. Here again, to get the strong convergence in $H^1_0(\Omega)\times L^2(\Omega)$, the idea is to use the energy identity and the convergence \eqref{Convergence-of-the-energy}. However, since the energy of the solutions of \eqref{Wave-eq-a-n} involves the matrix $A_n$, we should again be careful. We start by proving that the sequence $(A_n \nabla y_n(T))_{n \in \N}$ converges weakly in $L^2(\Omega)$ towards $A \nabla y(T)$. Indeed, $(A_n)_{n\in \N}$ strongly  converges to $A$ in $L^1(\Omega)$ and is uniformly bounded. Consequently,  $(A_n)_{n\in \N}$ strongly  converges to $A$ in $L^2(\Omega)$. Since the sequence $(\nabla y_n(T))_{n \in \N}$ converges weakly in $L^2(\Omega)$ to $\nabla y(T)$ at time $T$, we deduce that the sequence $(A_n \nabla y_n(T))_{n \in \N}$ converges weakly in $L^1(\Omega)$ towards $A \nabla y(T)$. Since this sequence is bounded in $L^2(\Omega)$, the sequence $(A_n \nabla y_n(T))_{n \in \N}$ actually converges to $A \nabla y(T)$ weakly in $L^2(\Omega)$. 

    To show the strong convergence of $(y_n(T), \partial_t y_n(T))_{n \in \N}$ to $(y(T), \partial_t y(T))$ in $H^1_0(\Omega) \times L^2(\Omega)$, we write:
    \begin{align}
        & \frac{1}{\alpha} \| \nabla y_n(T) - y(T)\|_{L^2(\Omega)}^2 
        +
        \| \partial_t y_n(T) - \partial_t y(T)\|_{L^2(\Omega)}^2
        \notag
        \\
        &\leq 
        \int_\Omega \left( A_n \nabla (y_n(T) - y(T)) \cdot \nabla (y_n(T) - y(T)) + |\partial_t y_n(T) - \partial_t y(T)|^2 \right)\ud x
        \notag
        \\
        & \leq 
    \int_\Omega \left( A_n \nabla y_n(T) \cdot \nabla y_n(T) + |\partial_t y_n(T)|^2 \right)\ud x
    - 
    2
    \int_\Omega \left( A_n \nabla y_n(T) \cdot \nabla y(T) + \partial_t y_n(T) \partial_t y(T) \right)\ud x
    \label{towards-conv-at-time-T}
   \\
   & \qquad + 
    \int_\Omega \left( A_n \nabla y(T) \cdot \nabla y(T) + |\partial_t y(T)|^2 \right)\ud x
    \notag
    \end{align}
    Lebesgue's dominated convergence theorem implies that
    $$
        \lim_{n \to \infty} \int_\Omega \left( A_n \nabla y(T) \cdot \nabla y(T) 
         + |\partial_t y(T)|^2 \right)\ud x
        =
        \int_\Omega \left( A \nabla y(T) \cdot \nabla y(T) 
         + |\partial_t y(T)|^2 \right)\ud x.
    $$
    The weak convergence of the sequence $(A_n \nabla y_n(T), \partial_t y_n(T))_{n \in \N}$ to $(A \nabla y(T), \partial_t y(T))$ in $ L^2(\Omega)\times L^2(\Omega)$ implies that 
    $$
        \lim_{n \to \infty} 
        \int_\Omega \left( A_n \nabla y_n(T) \cdot \nabla y(T) + \partial_t y_n(T) \partial_t y(T) \right)\ud x
        = 
        \int_\Omega \left( A \nabla y(T) \cdot \nabla y(T) 
         + |\partial_t y(T)|^2 \right)\ud x.  
    $$
    Combining the last two convergences with the convergence \eqref{Convergence-of-the-energy}, we deduce from \eqref{towards-conv-at-time-T} that the sequence $(y_n(T), \partial_t y_n(T))_{n \in \N}$ strongly converges to $(y(T), \partial_t y(T))$ in $H^1_0(\Omega) \times L^2(\Omega)$.
\end{proof}

We also introduce a corollary about the convergence of the normal trace that will be useful in the following:
\begin{corollary}
	\label{Cor-Conv-Normal-Trace}
	Under the assumptions of Theorem \ref{Thm-Regularization}. Let $\Gamma$ be a non-empty smooth open subset of $\partial \Omega$ and assume that, for all $n \in \N$, $A_n \in W^{1,\infty}(\Omega; \R^{d\times d})$ and  there exists a neighborhood $\mathcal{V}$ of $\Gamma$ in $\overline\Omega$ such that the sequence $(A_n)_{n \in \N}$ is uniformly bounded in $W^{1, \infty} (\mathcal{V};\R^{d\times d})$. Then the sequence $(\partial_\nu y_n|_{(0,T)\times \Gamma})_{n \in \N}$ is strongly convergent to $\partial_\nu y|_{(0,T)\times \Gamma}$ in $L^2((0,T)\times \Gamma)$, implying in particular that $\partial_\nu y|_{(0,T)\times \Gamma} \in L^2((0,T)\times \Gamma)$.
\end{corollary}
Note that before proving Corollary \ref{Cor-Conv-Normal-Trace}, it is important to know what is the definition of $\partial_\nu y$ for $y$ solution of \eqref{Wave-eq-a}. We define it in a weak manner as a natural extension of the formulas we would get for smooth coefficients:
\begin{definition}
	\label{Def-normal-derivative-y}
	Let $\Omega$ be a smooth bounded domain of $\R^{d}$, $d \geq 1$, and let $T>0$.
	Let $A \in L^\infty(\Omega; \R^{d\times d})$ taking value in the set of symmetric matrices such that \eqref{Ass-WP-a} holds for some $\alpha>0$. Then for $(y_0, y_1) \in H^1_0(\Omega) \times L^2(\Omega)$, we define $\partial_\nu y$ as the unique element of $(H^{1/2}(0,T; L^2(\partial \Omega)) \cap L^2(0,T; H^{1/2}(\partial\Omega)))'$ such that for all $\varphi \in H^1(0,T;L^2(\Omega)) \cap L^2(0,T; H^1(\Omega))$, 
	\begin{multline}
		\label{Def-Weak-Derivative}
		\langle \partial_\nu y, \varphi \rangle_{
        (H^{1/2}(0,T; L^2(\partial \Omega)) \cap L^2(0,T; H^{1/2}(\partial\Omega)))', 
        (H^{1/2}(0,T; L^2(\partial \Omega)) \cap L^2(0,T; H^{1/2}(\partial\Omega)))}
		\\
		= 
		\int_\Omega \partial_t y(t, \cdot) \varphi(t, \cdot)\, \ud x\Big|_0^T
		- \int_0^T \int_\Omega \partial_t y \partial_t \varphi \, \ud x \, \ud t
		+ 
		\int_0^T \int_\Omega A \nabla y \cdot \nabla \varphi\, \ud x\, \ud t.
	\end{multline}
\end{definition}
In Definition \ref{Def-normal-derivative-y}, we used that the set of traces of functions in $H^1(0,T;L^2(\Omega)) \cap L^2(0,T; H^1(\Omega))$ exactly is  $ H^{1/2}(0,T; L^2(\partial \Omega)) \cap L^2(0,T; H^{1/2}(\partial\Omega))$ (see \cite[Chapter 4, Section 2]{Lions-Magenes-II}).

We emphasize that the definition \eqref{Def-normal-derivative-y} is independent from the lifting $\varphi_g$ used to lift $g$: Indeed, if $\varphi_g^a$ and $\varphi_g^b$ are two liftings of $g$, then $ \varphi_g^a - \varphi_g^b \in  H^1(0,T;L^2(\Omega)) \cap L^2(0,T; H^1_0(\Omega))$, and by the equation \eqref{Wave-eq-a}, 
$$
	0= 
		\int_\Omega \partial_t y(t, \cdot) (\varphi_g^a - \varphi_g^b)(t, \cdot)\, \ud x\Big|_0^T - \int_0^T \int_\Omega \partial_t y \partial_t (\varphi_g^a - \varphi_g^b) \, \ud x \, \ud t
		+ 
		\int_0^T \int_\Omega A \nabla y \cdot \nabla (\varphi_g^a - \varphi_g^b)\, \ud x\, \ud t.
$$
\begin{proof}[Proof of Corollary \ref{Cor-Conv-Normal-Trace}]
	The proof is based on the classical multiplier method, which consists in multiplying  \eqref{Wave-eq-a-n} by  $m \cdot \nabla y_n$, for a smooth $(\mathscr{C}^1)$ vector field $m = m(x)$ compactly supported in $\mathcal{V}$ such that $m \cdot \nu = 1$ in a neighborhood of $\Gamma$  and $m \cdot \nu \geq 0$ on $\partial \Omega$ (such a vector field exists, see for instance \cite[Chapter 1, Lemma 3.1]{lionsHUM}). Indeed, the computations give that for all $n \in \N$, and $(y_0, y_1) \in H^2 \cap H^1_0(\Omega) \times H^1_0(\Omega)$, the solution $y_n$ of \eqref{Wave-eq-a-n} satisfies
	\begin{align*}
		0 
		&= \int_0^T \int_\Omega (\partial_{tt} y_n  - \div (A_n \nabla y_n)) (m \cdot \nabla y_n)\, \ud x \, \ud t
		\\
		& = 
		\int_\Omega \partial_t y_n \, m \cdot \nabla y_n \, \ud x\Big|_0^T
		+ 
		\frac{1}{2} \int_0^T \int_\Omega \div(m) |\partial_t y_n |^2 \, \ud x \, \ud t
		\\
		& \qquad - 
		\int_0^T \int_{\partial\Omega} |\partial_\nu y_n |^2 (A _n\nu \cdot \nu) (m\cdot \nu) \, \ud \sigma\, \ud t
		+
		\int_0^T \int_\Omega A_n \nabla y_n \cdot \nabla (m \cdot \nabla y_n) \, \ud x \, \ud t.
	\end{align*}
	Note that here, we used that $A_n \in W^{1,\infty}(\Omega; \R^{d\times d})$ to get that $H^2 \cap H^1_0(\Omega)$ is the domain of the operator $- \div (A_n \nabla \cdot)$ with Dirichlet boundary condition in $L^2(\Omega)$, and thus that all the computations above are justified.

	Using the fact that $A_n$ is symmetric, for $z \in H^2(\Omega) \cap H^1_0(\Omega)$,
	\begin{align*}
		& \int_\Omega A_n \nabla z \cdot \nabla (m \cdot \nabla z) \, \ud x
		= 
		\int_\Omega A_n \nabla (m \cdot \nabla z) \cdot \nabla z  \, \ud x
		\\ 
		&
		= 
		\frac{1}{2} \int_\Omega \left( A_n  \nabla z \cdot \nabla (m \cdot \nabla z) + A_n \nabla (m \cdot \nabla z) \cdot \nabla z  \right) \, \ud x
		\\ 
		& 
		= 
		\frac{1}{2} \sum_{i,j,k= 1}^d \int_\Omega a_{i,j,n} \left(  \partial_i z  \partial_j (m_k \partial_k z) + \partial_i (m_k \partial_k z)  \partial_j z  \right) \, \ud x
		\\ 
		& 
		= 
		\frac{1}{2} \sum_{i,j,k= 1}^d \int_\Omega a_{i,j,n} m_k \partial_k  \left(  \partial_i z  \partial_j  z  \right) \, \ud x 
		+ 
		\frac{1}{2} \sum_{i,j,k= 1}^d \int_\Omega a_{i,j,n} \left(  \partial_i z \partial_k z \partial_j m_k  + \partial_i m_k \partial_k z  \partial_j z  \right) \, \ud x
		\\
		& = 
		- \frac{1}{2} \int_\Omega (\div (m) A_n \nabla z \cdot \nabla z + m \cdot \nabla A_n \nabla z \cdot \nabla z ) \, \ud x 
		+ 
		\frac{1}{2} \int_{\partial \Omega} (\partial_\nu z)^2 (A_n \nu \cdot \nu) (m\cdot \nu)\, \ud\sigma \,\ud t
		\\
		& \qquad + 
		 \sum_{i,j,k= 1}^d \int_\Omega a_{i,j,n}  \partial_i z \partial_k z \partial_j m_k \, \ud x.
	\end{align*}
	It follows that, for all $y_n$ solutions of \eqref{Wave-eq-a-n} with data $(y_0, y_1) \in H^2(\Omega) \cap H^1_0(\Omega) \times H^1_0(\Omega)$, 
	\begin{multline}
		\label{Multiplier-Identity-m-a-n}
		\frac{1}{2}\int_0^T \int_{\partial\Omega} |\partial_\nu y_n |^2 (A _n\nu \cdot \nu) (m\cdot \nu) \, \ud \sigma\, \ud t
		 = 
		\int_\Omega \partial_t y_n \, m \cdot \nabla y_n \, \ud x\Big|_0^T
		+ 
		\frac{1}{2} \int_0^T \int_\Omega \div(m) |\partial_t y_n |^2 \, \ud x \, \ud t
		\\
		- \frac{1}{2} \int_0^T \int_\Omega (\div (m) A_n \nabla y_n \cdot \nabla y_n + m \cdot \nabla A_n \nabla y_n \cdot \nabla y_n ) \, \ud x \, \ud t
		\\
		+ 
		\sum_{i,j,k= 1}^d \int_0^T \int_\Omega a_{i,j,n}  \partial_i y_n \partial_k y_n \partial_j m_k \, \ud x\, \ud t.	
	\end{multline}
	Also note that by density of $H^2(\Omega)\cap H^1_0(\Omega) \times H^1_0(\Omega)$ in $H^1_0(\Omega) \times L^2(\Omega)$, and the fact that all the quantities in the right-hand side depend continuously of the initial data for the topology of $H^1_0(\Omega) \times L^2(\Omega)$, we easily deduce that this identity holds true for initial data in $H^1_0(\Omega) \times L^2(\Omega)$.

	We now show a similar identity for solutions of the limit equation \eqref{Wave-eq-a}. Before doing that, let us remark that the assumptions imply that $(m\cdot\nabla) A \in L^\infty(\mathcal{V};\R^{d\times d})$. We then establish a similar identity as \eqref{Multiplier-Identity-m-a-n} for initial data $(y_0 , y_1) \in \mathcal{D}(\mathscr{A}) \times H^1_0(\Omega)$, where $\mathscr{D}(\mathscr{A}) = \{ y \in H^1_0(\Omega), \, \div ( A \nabla y ) \in L^2(\Omega)\}$, endowed with the norm $\|  \div ( A \nabla \cdot )\|_{L^2(\Omega)} + \|\cdot \|_{L^2(\Omega)}$. In such case, classical semigroup theory implies that the solution $y$ of \eqref{Wave-eq-a} satisfies $y \in \mathscr{C}^0([0,T]; \mathscr{D}(\mathscr{A})) \cap \mathscr{C}^1([0,T]; H^1_0(\Omega))$. Besides, since $A$ is $W^{1,\infty}(\mathcal{V})$, local elliptic regularity results imply that $y \in \mathscr{C}^0([0,T]; H^2 (\Supp(m)))$. It follows that $m \cdot \nabla y$ belongs to $L^2(0,T; H^1(\Omega)) \cap H^1(0,T; H^1(\Omega))$. We can thus do the same computations as for \eqref{Multiplier-Identity-m-a-n}, and we obtain
	\begin{multline}
		\label{Multiplier-Identity-m-a}
		\frac{1}{2}\int_0^T \int_{\partial\Omega} |\partial_\nu y |^2 (A \nu \cdot \nu) (m\cdot \nu) \, \ud \sigma\, \ud t
		 = 
		\int_\Omega \partial_t y \, m \cdot \nabla y \, \ud x\Big|_0^T
		+ 
		\frac{1}{2} \int_0^T \int_\Omega \div(m) |\partial_t y |^2 \, \ud x \, \ud t
		\\
		- \frac{1}{2} \int_0^T \int_\Omega (\div (m) A \nabla y \cdot \nabla y + m \cdot \nabla A \nabla y \cdot \nabla y ) \, \ud x \, \ud t
		+ 
		\sum_{i,j,k= 1}^d \int_0^T \int_\Omega a_{i,j}  \partial_i y \partial_k y \partial_j m_k \, \ud x\, \ud t.	
	\end{multline}
	Here again, we can use the density of $\mathscr{D}(\mathscr{A})\times H^1_0(\Omega)$ in $H^1_0(\Omega) \times L^2(\Omega)$, and the fact that all the quantities in the right-hand side depend continuously of the initial data for the topology of $H^1_0(\Omega) \times L^2(\Omega)$, to deduce that this identity holds true for all initial data in $H^1_0(\Omega) \times L^2(\Omega)$.

	Let us now fix $(y_0, y_1) \in H^1_0(\Omega) \times L^2(\Omega)$. Since the right-hand side of the identity \eqref{Multiplier-Identity-m-a-n} is bounded uniformly in $n$ from the convergences in Theorem \ref{Thm-Regularization}, the sequence $(m\cdot \nu \partial_\nu y_n|_{(0,T) \times \partial\Omega})_{n \in \N}$ is uniformly bounded in $L^2((0,T) \times \partial \Omega)$ and compactly supported in $\Supp \{ m\cdot \nu\} \cap \partial\Omega$, so that it weakly converges to some $\mu \in L^2((0,T)\times \Omega)$ compactly supported in $\Supp \{ m\cdot \nu\} \cap \partial\Omega$. 
	
	We then check that $\mu$ coincides with $m\cdot \nu \partial_\nu y|_{(0,T) \times \partial\Omega}$ for $y$ solution of \eqref{Wave-eq-a}. Writing the definition of $\partial_\nu y_n$ in \eqref{Def-Weak-Derivative}, and using the strong convergence of $(\partial_t y_n)_{n \in \N}$ and $(\nabla y_n)_{n \in \N}$ towards $\partial_t y$ and $\nabla y$ in $L^2((0,T) \times \Omega)$ and the weak-$\star$ $L^\infty(\Omega)$ convergence of $(A_n)_{n \in \N}$ towards $A$, we immediately get that $(\partial_\nu y_n|_{(0,T) \times \partial\Omega})_{n \in \N}$ converges to $\partial_\nu y|_{(0,T) \times \partial\Omega}$ in $(H^{1/2}(0,T; L^2(\partial \Omega)) \cap L^2(0,T; H^{1/2}(\partial\Omega)))'$. It thus follows immediately that $\mu$ coincides with $m\cdot \nu \partial_\nu y |_{(0,T) \times \Omega}$, implying in particular $\partial_\nu y|_{(0,T) \times \Gamma} \in L^2((0,T)\times \Gamma)$.
	
	It remains to check that the sequence $(m\cdot \nu \partial_\nu y_n|_{(0,T) \times \partial\Omega})_{n \in \N}$ converges strongly in $L^2((0,T)\times \partial\Omega)$ to $m\cdot \nu \partial_\nu y|_{(0,T) \times \partial\Omega}$. In order to do that, we simply check that we can pass to the limit in the identity \eqref{Multiplier-Identity-m-a-n} as $n\to \infty$, and recover the right hand side of \eqref{Multiplier-Identity-m-a}, thanks to the weak-$\star$ $W^{1, \infty}(\mathcal{V}; \R^{d\times d})$ convergence of $(A_n)_{n \in \N}$ towards $A$ and the strong convergences obtained in Theorem \ref{Thm-Regularization}.
\end{proof}

In the context of the plate equations, we can get a similar result, that we state below for spatial operators of the form $- \Delta (a \Delta \cdot )$ for sake of simplicity\footnote{A similar statement could be proven for operators $L$ of the form $L = \sum_{i,j,k,\ell} \partial_i \partial_j( a_{i, j,k,\ell} \partial_k \partial_\ell \cdot)$, in which $A = (a_{i,j,k,\ell})_{(i,j,k,\ell) \in \{1, \cdot, d\}^4}$ is a tensor of order $4$ such that the form $$ {\bf a}(u,v) = \int_\Omega \sum_{i,j,k,\ell} a_{i,j,k,\ell} \, \partial_{i} \partial_j u\,  \partial_k \partial_\ell v\, \ud x$$ is symmetric, coercive and bounded on $H^2_0(\Omega)$, but we intend to keep our exposition technically simple for readability.}
\begin{theorem}
	\label{Thm-Regularization-Plate}
	Let $\Omega$ be a smooth bounded domain of $\R^d$, $d \geq 1$, and let $T>0$.
	
	Let $a \in L^\infty(\Omega)$ and a sequence $(a_n)_{n \in \N} \in L^\infty(\Omega )$ such that, for some $\alpha >0$, 
	\begin{align*}
		&
		\text{a.e. } x \in \Omega, \quad \frac{1}{\alpha} \leq a(x) \leq \alpha, 
		\quad \text{ and } \quad 
		 \forall n \in \N, \, \text{a.e. } x \in \Omega, \quad \frac{1}{\alpha} \leq a_n(x) \leq \alpha, 
		\\
		\label{Convergence-a-n-a-plates}
		& (a_n) \underset{n \to \infty} \longrightarrow a \text{ in } L^1(\Omega), 
		\hbox{ and } 
		a_n \underset{n \to \infty} \longrightarrow a \text{ almost everywhere in } \Omega. 
	\end{align*}

	Then, for any $(y_0, y_1) \in H^2_0(\Omega) \times L^2(\Omega)$, denoting respectively by $y$ and $y_n$ the solutions of 
	\begin{equation}
		\label{Plate-eq-a}
		\left\{
			\begin{array}{ll}
				\partial_{tt} y + \Delta (a \Delta y) = 0 & \text{ in } (0,T) \times \Omega, 
				\\
				y = \partial_\nu y = 0 & \text{ on } (0,T) \times \partial \Omega, 
				\\ 
				(y, \partial_t y)\Big|_{t= 0} = (y_0, y_1) & \text{ in } \Omega,
			\end{array}
		\right.
	\end{equation}
	respectively, 
	\begin{equation}
		\label{Plate-eq-a-n}
		\left\{
			\begin{array}{ll}
				\partial_{tt} y_n + \Delta (a_n \Delta y_n) = 0 & \text{ in } (0,T) \times \Omega, 
				\\
				y_n = \partial_\nu y_n = 0 & \text{ on } (0,T) \times \partial \Omega, 
				\\ 
				(y_n, \partial_t y_n)\Big|_{t= 0} = (y_0, y_1) & \text{ in } \Omega,
			\end{array}
		\right.
	\end{equation}
	the sequence $(y_n)_{n \in \N}$ converges strongly to $y$ in $L^2(0,T; H^2_0(\Omega)) \cap H^1(0,T; L^2(\Omega))$ as $n \to \infty$.
	
	Besides, if for all $n \in \N$, $a_n \in W^{1, \infty}(\Omega)$ and if $\Gamma$ is a non-empty smooth open subset of $\partial \Omega$ and there exists a neighborhood $\mathcal{V}$ of $\Gamma$ in $\overline\Omega$ such that the sequence $(a_n)_{n \in \N}$ is uniformly bounded in $W^{1, \infty} (\mathcal{V})$, the sequence $(\Delta y_n|_{(0,T)\times \Gamma})_{n \in \N}$ is strongly convergent to $\Delta y|_{(0,T)\times \Gamma}$ in $L^2((0,T)\times \Gamma)$, implying in particular that $\Delta y|_{(0,T)\times \Gamma} \in L^2((0,T)\times \Gamma)$.
\end{theorem}

The proof is left to the reader as it is a minor variation of the proof of Theorem \ref{Thm-Regularization} and Corollary \ref{Cor-Conv-Normal-Trace}.

\section{Multiplier methods} \label{MultSect}

\subsection{Preliminary result: how to regularize the coefficients?}

As our approach to prove Theorems \ref{MainThmOBS} and \ref{MainThmOBS-distributed} (respectively Theorem \ref{MainThmOBS-Plates}) is based on the regularization result of Theorem \ref{Thm-Regularization} (respectively Theorem \ref{Thm-Regularization-Plate}), it is obvious that a key element will be given by the way we regularize the coefficients $A$ for the wave equation \eqref{Wave-eq-intro} (respectively $a$ for the plate equation \eqref{Plate-eq-a}). 

The main point is that we can regularize these coefficients in such a way that the multiplier condition \eqref{ObsCond-a}, which is satisfied only in the sense of distributions for $A$ (resp. $a$), will be satisfied pointwise for the regularized versions of $A$. 

We start by recalling the following approximation result from \cite[Proposition 2.1]{Dehman-Erv-2016} in the context of a wave equation with space varying density $\rho$.

%
%\red {Maybe Propositions 3.1 and 3.2 go to previous Section?}
%

\begin{proposition}
	\label{Prop-Approx-Rho}
	Let $\Omega_0$ be a smooth domain containing $\overline\Omega$ and let $\rho$ satisfy the assumptions 	%
	\begin{itemize}
		\item The density $\rho$ is strictly positive and bounded in $\Omega_0$: there exist $\rho_1>0$ and $\rho_2 >0$ such that 
		\begin{equation}
			\label{PositivityRho}
			\forall x \in  \Omega_0,\quad \rho_1 \leq \rho(x)\leq \rho_2.  		
		\end{equation}
		\item The density $\rho$ is continuous in $\overline{\Omega_0}$:
		\begin{equation}
			\label{ContinuityRho}
			\rho \in \mathscr{C}^0(\overline{\Omega_0}), 
		\end{equation}
	\end{itemize}
	and
	\begin{equation}
		\label{ObsCondRho}	
		\exists \alpha \in (0, 2], 
		\hbox{ such that } 
		\quad 
		x\cdot \nabla\rho(x) + (2 -\alpha) \rho(x) \geq 0
		\text{ in the sense of } \mathscr{D}'(\Omega_0), 
	\end{equation} 
	and further assume 
	\begin{equation}
		\label{Cond-on-0}
		\left\{
			\begin{array}{l}
				0 \notin \overline{\Omega_1}, 
				\\
				\hbox{or}
				\\
				\hbox{$\rho$ is $\mathscr{C}^1$ in a neighborhood of $0$.}
			\end{array}
		\right.  
	\end{equation}
	\\
	Let $\eta$ be a real valued non negative smooth function on  $\R^{d}$ supported in the unit ball $B(0,1)$ and such that 
	\begin{equation}
		\label{NormalizationEta}
		\int_{\R^d} \eta(x)\, \ud x=1 \hbox{ and }\, \, \forall i \in \{1, \cdots, d\}, \, \int_{\R^d} x_i \eta(x) \, \ud x = 0.
	\end{equation}
	Then there exists $\varepsilon_0>0$ such that the functions $(\rho_\varepsilon)_{\varepsilon \in (0, \varepsilon_0)}$ defined on $\overline\Omega$ by 
	\begin{equation}	
		\label{Def-rho-eps}
		\left\{	
			\begin{array}{l}
				\ds \forall x \in \overline\Omega \setminus\{0\},
				\quad
				\rho_\varepsilon(x)
				= 
				\frac{1}{(\varepsilon |x|)^{d}} \int_{\Omega_0} \rho(y) \eta\left(\frac{x-y}{\varepsilon \vert x \vert}\right)\,\ud y,
				\smallskip \\
				\ds\hbox{If } 0 \in \overline\Omega, \quad \rho_\varepsilon(0)
				= 
				\rho(0),
			\end{array}
		\right.
	\end{equation}
	satisfies the following properties:
	\begin{itemize}
		\item[(i)] For each $\varepsilon \in (0, \varepsilon_0)$, $\rho_\varepsilon$ belongs to $\mathscr{C}^1(\overline\Omega)$;
		\item[(ii)] For each $\varepsilon \in (0, \varepsilon_0)$, $\rho_\varepsilon$ satisfies conditions \eqref{PositivityRho}.
		\item[(iii)] The sequence $\rho_\varepsilon$ strongly converges in $L^\infty(\Omega)$ to $\rho$ as $\varepsilon \to 0$:
		\begin{equation}
			\label{StrongConvergence-rho}
			\lim_{\varepsilon \to 0} \norm{ \rho_\varepsilon - \rho}_{L^\infty(\Omega)} = 0.
		\end{equation}
		\item[(iv)] For each $\varepsilon \in (0, \varepsilon_0)$, condition \eqref{ObsCondRho} is satisfied pointwise in $x \in \Omega$, i.e.  
		\begin{equation}
			\label{ObsCondRhoEps}
			\forall x \in \Omega, 
			\quad
			x\cdot \nabla\rho_\varepsilon(x) + (2 -\alpha) \rho_\varepsilon(x) \geq 0.
		\end{equation}
	\end{itemize}
\end{proposition}

We claim that the same arguments as the ones developed in the proof of \cite[Proposition 2.1]{Dehman-Erv-2016} can be adapted as follows: 
\begin{proposition}
	\label{Prop-Approx-a}
	Let $\Omega_0$ be a smooth open domain containing $\overline\Omega$ and let $A \in L^\infty (\Omega_0; \R^{d\times d})$ satisfying Assumption \ref{Assumption1}.
      %
    %%%  Hypothesis  de a copiadas arriba, y por lo tanto comentadas aqui:
	% \begin{itemize}
	% 		\item The velocity $a$ is strictly positive and bounded in $\overline\Omega$: there exist $\alpha_1>0$ and $\alpha_2 >0$ such that 
		% 		\begin{equation}\label{Positivity-a}
%			\forall x \in  \Omega,\quad 0<\alpha_1 \leq a(x)\leq \alpha_2.  		
%		\end{equation}
		%
%		\item The velocity $a$ is bounded in $\Omega_1$:
		%
%		\begin{equation}
%			\label{Bounded-Omega-1-a}
%			a \in L^\infty(\Omega_1).
%		\end{equation}
		%
%	\end{itemize}
	%
%	and there exists $ \red{  \beta > -2 } $ such that 
	%
%	\begin{equation}		\label{ObsCond-a}	
		%\exists \beta \in \R,  CORRECTED:		\hbox{ such that } \,  
		% x\cdot \nabla a  + \beta a  \leq 0
		%\text{ in the sense of } \mathscr{D}'(\Omega_1), 
	 %  \end{equation} 
	% in the sense of $  \mathscr{D}'(\Omega_1)$. 
	%Further assume 
	%
	%\begin{equation}
	%	\label{Cond-on-0-a}
	%	\left\{
	%		\begin{array}{l}
	%			0 \notin \overline{\Omega_1}, 
	%			\\
%				\hbox{or}
%				\\
%				\hbox{$a$ is $C^1$ in a neighborhood of $0$.}
%			\end{array}
%		\right.  
%	\end{equation}
	% \\
	%
	Let $\eta$ be a real valued non negative smooth function on  $\R^{d}$ supported in the unit ball $B(0,1)$ and satisfying \eqref{NormalizationEta}.

	Then there exists $\varepsilon_0>0$ such that the functions $(A_\varepsilon)_{\varepsilon \in (0,\varepsilon_0)}$ defined on $\overline\Omega$ by 
	\begin{equation}	
		\label{Def-a-eps}
		\left\{	
			\begin{array}{l}
				\ds \forall x \in \overline\Omega \setminus\{0\},
				\quad
				A_\varepsilon(x)
				= 
				\frac{1}{(\varepsilon |x|)^{d}} \int_{\Omega_0} A(y) \eta\left(\frac{x-y}{\varepsilon \vert x \vert}\right)\, \ud y,
				\smallskip \\
				\ds\hbox{If } 0 \in \overline\Omega, \quad A_\varepsilon(0)
				= 
				A(0),
			\end{array}
		\right.
	\end{equation}
	satisfies the following properties:
	\begin{itemize}
		\item[(i)] For each $\varepsilon \in (0, \varepsilon_0)$, $A_\varepsilon$ belongs to $W^{1,\infty}(\Omega; \R^{d \times d})$ and to $\mathscr{C}^1(\overline\Omega\setminus\{0\}; \R^{d\times d})$;

		\item[(ii)] For each $\varepsilon \in (0, \varepsilon_0)$, $A_\varepsilon$ satisfies:
		\begin{equation}
			\label{Positivity-A-eps}
			\forall x \in \overline\Omega,\, \forall \varepsilon \in (0, \varepsilon_0), \quad \frac{1}{\alpha} I_{d \times d} \leq A_\varepsilon(x) \leq \alpha I_{d \times d}, 
		\end{equation}
		\item[(iii)] The sequence $A_\varepsilon$ strongly converges in $L^1(\Omega; \R^{d\times d})$ to $A$ as $\varepsilon \to 0$:
		\begin{equation}
			\label{StrongConvergence}
			\lim_{\varepsilon \to 0} \norm{A_\varepsilon - A}_{L^1(\Omega; \R^{d\times d})} = 0.
		\end{equation}
		\item[(iv)] For each $\varepsilon \in (0, \varepsilon_0)$, condition \eqref{ObsCond-a} is satisfied pointwise in $x \in \Omega$, i.e.  
		\begin{equation}
			\label{ObsCond-a-eps}
			\forall \xi \in \R^d, \, \forall x \in \overline{\Omega}, 
			\quad
			\left(x\cdot \nabla A_\varepsilon(x) - \beta A_\varepsilon(x)) \xi\right) \cdot \xi \leq 0.
		\end{equation}
		\item[(v)] If $A \in W^{1,\infty}(\mathcal{V}; \R^{d \times d})$ for an open set $\mathcal{V}$ included in $\Omega_0$, for any open set $\mathcal{V}_1$ such that $\overline{\mathcal{V}_1} \subset \mathcal{V}$, there exists $\varepsilon_1 \in (0, \varepsilon_0)$ such that $\sup_{\varepsilon \in (0, \varepsilon_1)} \|\nabla A_\varepsilon\|_{L^\infty(\mathcal{V}_1)} < \infty$.
	\end{itemize}
\end{proposition}

\begin{remark}
	\label{Rk-Approx-For-Plates}
	Similarly, if $a \in L^\infty(\Omega)$ satisfies the assumption of Theorem \ref{MainThmOBS-Plates}, then the functions $(a_\varepsilon)_{\varepsilon \in (0,\varepsilon_0)}$ (defined as in \eqref{Def-a-eps} with $a$ instead of $A$) satisfy all the items of Proposition \ref{Prop-Approx-a} with $a_\varepsilon$ and $a$ instead of $A_\varepsilon$ and $A$, and item $(iv)$ with $\beta$ as in \eqref{Multiplier-Condition-Plate}. This can be deduced immediately by considering $A = a I_{d \times d}$ in Proposition \ref{Prop-Approx-a}, and remarking that the proof of item $(iv)$ gives that $x \cdot \nabla a_\varepsilon - \beta a_\varepsilon \leq 0$ in $\mathscr{D}'(\Omega_0)$ implies $x \cdot \nabla a - \beta a \leq 0$ in $\overline\Omega$, no matter what $\beta \in \R$ is.
\end{remark}

The construction of the regularized approximations \( A_\varepsilon \), as defined in \eqref{Def-a-eps}, is directly adapted from the one in Proposition \ref{Prop-Approx-Rho}. We still give its proof for convenience. 

\begin{proof}
	Before starting the proof, let us note that, for $\varepsilon>0$ sufficiently small to get $\varepsilon \sup_{x \in \Omega} |x| < d(\Omega, \partial \Omega_0)$,  the formula  \eqref{Def-a-eps} can be rewritten as
	\begin{equation}
		\label{Other-Formula-For-A-eps}
	 	\forall x \in \overline{\Omega} \setminus \{0\}, \quad A_\varepsilon(x) = \int_{\R^d} A(x- \varepsilon |x| y) \eta (y) \, \ud y.
	\end{equation}
	In the proof below, we will always assume that $\varepsilon < d(\Omega, \partial \Omega_0)/\sup_{x \in \Omega} |x| $, so that this formula can be used. 
	\medskip\\	
	{\it Proof of item (i).} The proof of the $\mathscr{C}^1$ regularity of $A_\varepsilon$ in $\overline{\Omega} \setminus\{0\}$ is straightforward, since it is an immediate consequence of the smoothness of $\eta$ by using formula \eqref{Def-a-eps}. The fact that $\nabla A_\varepsilon$ is bounded in $\Omega$ is also clear from formula \eqref{Def-a-eps} when $0 \notin \overline\Omega$.
	\\
	When $0$ belongs to $\overline\Omega$, the situation is slightly more intricate. We first remark that in fact formula \eqref{Def-a-eps} actually defines $A_\varepsilon$ in an open neighborhood of $\overline\Omega$, so that $0$ can be considered as an interior point of $\Omega$ without loss of generality by slightly enlarging the set $\Omega$ if needed. 
	
	In order to check that $A_\varepsilon$ is differentiable at $0$ when $0$ belongs to $\Omega$, we use the fact in such case, by Assumption \ref{Assumption1}, $A$ is Lipschitz in a neighborhood of $0$, say $B(0,r_0)$ for some $r_0 >0$. In particular, $A$ is continuous in $B(0,r_0)$. We can then easily obtain the Lipschitz continuity of $A_\varepsilon$ in $B(0,r_0/2)$, for 
    $\varepsilon < 1$,
    % $\varepsilon r_0/2 \leq r_0$  <---- previous hypothesis
    by writing, for $x_1$ and $x_2$ in $B(0,r_0/2)$,
	\begin{align*}
		\| A_\varepsilon(x_1) - A_\varepsilon(x_2) \|
		& \leq 
		\left|
			 \int_{\R^d} (A(x_1- \varepsilon |x_1| y) - A(x_2 - \varepsilon |x_2| y) \eta (y) \, \ud y
		\right| 
		\\
		& 
		\leq 
		\int_{\R^d} | A(x_1- \varepsilon |x_1| y) - A(x_2 - \varepsilon |x_2| y) |\eta (y) \, \ud y
		\\
		& 
		\leq 
		\| \nabla A\|_{L^\infty(B(0,r_0))}
		\int_{\R^d} |x_1 - x_2|  (1+  \varepsilon |y|) \eta (y) \, \ud y
		\\ 
		& \leq 
		C_\eta (1+ \varepsilon) \| \nabla A\|_{L^\infty(B(0,r_0))}
		|x_1 - x_2| , 
	\end{align*}
	where we used that $\eta$ is compactly supported in the unit ball $B(0,1)$. It follows that $A_\varepsilon$ is Lipschitz in $B(0,r_0/2)$. Since it is also clear from formula \eqref{Def-a-eps} that  $\nabla A_\varepsilon$ is bounded in $\overline\Omega \setminus B(0,r_0/2)$, we get that $A_\varepsilon$ is Lipschitz in $\Omega$ as announced.
	\medskip\\
	{\it Proof of item (ii).} To prove item (ii), we simply remark that from \eqref{NormalizationEta} for all $x \in \R^d$,
	$$
		\int_{\R^d} \eta\left(\frac{x-y}{\varepsilon \vert x \vert}\right)\, \ud y 
		= 
		(\varepsilon |x|)^d.
	$$
	One then easily gets that $A_\varepsilon$ satisfies \eqref{Positivity-A-eps} on $\overline\Omega$ using \eqref{Positivity-A}, provided $\varepsilon \in (0, \varepsilon_0)$.
	\medskip\\	
	{\it Proof of item (iii).} Note that $A \in L^\infty(\Omega_0; \R^{d\times d})$. Thus, for all $\rho >0$, there exists $\tilde A_\rho \in \mathscr{C}_c(\Omega_0;  \R^{d \times d})$ such that $\| A - \tilde A_\rho \|_{L^1(\Omega_0)} \leq \rho $.
	
	Now, since  $\tilde A_\rho \in \mathscr{C}_c(\Omega_0;  \R^{d \times d})$, it is uniformly continuous on $\overline{\Omega_0}$: For every $\delta >0$, there exists $\gamma_\rho(\delta) >0$ such that 
	\begin{equation}
		\label{UniformCont}
	  	\forall (x,y) \in \overline{\Omega_0}^2, 
		\hbox{ with }
		| x-y| \leq \gamma_\rho(\delta), \quad | \tilde A_\rho(x)-\tilde A_\rho(y)| \leq \delta.
	\end{equation}
	We then remark that 
	\begin{align*}
		\forall x \in \overline\Omega, 
		\quad
		& A_\varepsilon(x)-A(x)
		 = 
		\int_{\R^d} (A(x - \varepsilon |x| y) - A(x)) \eta (y) \ud y
%		\\
%		& = 
%		\int_{\R^d} (A(x - \varepsilon |x| y) - \tilde A(x - \varepsilon |x| y)) \eta (y) \ud y
%		+ 
%		\int_{\R^d} (\tilde A(x - \varepsilon |x| y) - \tilde A(x)) \eta (y) \ud y		
%		+ 
%		\int_{\R^d} (\tilde A(x) - A(x)) \eta (y) \ud y
		\\ 
		&= 
		\int_{\R^d} (A(x - \varepsilon |x| y) - \tilde A_\rho(x - \varepsilon |x| y)) \eta (y) \ud y
		+ 
		\int_{\R^d} (\tilde A_\rho(x - \varepsilon |x| y) - \tilde A_\rho(x)) \eta (y) \ud y		
		+
		(\tilde A_\rho(x) - A(x)).
	\end{align*}
	Therefore, using \eqref{NormalizationEta}, for all $\delta >0$, taking $\varepsilon$ small enough to guarantee $\varepsilon \sup_{x \in \Omega} |x| \leq \gamma_\rho (\delta)$,  
	$$	
		\int_\Omega | A_\varepsilon(x)-A(x) | \ud x
		\leq 
		2 \| A - \tilde A_\rho \|_{L^1(\Omega_0)} + \int_{\Omega_0} \sup_{y \in B(0,1)} | \tilde A_\rho(x - \varepsilon |x| y) - \tilde A_\rho(x))| \, \ud x
		\leq 
		2 \rho + \delta | \Omega_0| .
	$$
	Accordingly, for any $\delta >0$, 
	$$
		\limsup_{\varepsilon \to 0} \| A_\varepsilon - A\|_{L^1(\Omega)} \leq 2 \rho + \delta |\Omega_0|. 
	$$
	Consequently, 
	$$
		\limsup_{\varepsilon \to 0} \| A_\varepsilon - A\|_{L^1(\Omega)} \leq 2 \rho. 
	$$
	As $\rho$ is also arbitrary, we have obtained that $A_\varepsilon$ converges to $A$ in $L^1(\Omega; \R^{d\times d})$.
%	
%	%
%	Therefore, if we set $R_1 = \max\{|x|, \hbox{ for } x \in \overline{\Omega_1} \}$, for all $\beta >0$, as soon as $\varepsilon R_1 \leq \gamma(\beta)$ given above in \eqref{UniformCont}, we obtain $\vert\rho_\varepsilon(x)-\rho(x)\vert \leq \beta$  for every $x \in \overline{\Omega}$ providing $\varepsilon < \min 
%	\{\varepsilon_0,\gamma(\beta)/R_1\}$. This concludes the proof of item (iii).
	%
	\medskip\\
	{\it Proof of item (iv).}  We first notice that for $\varepsilon>0$ small enough, formula \eqref{Def-a-eps} gives that for all $x \in \overline\Omega \setminus\{0\}$, 
	$$
	A_\varepsilon(x)
				= 
				\frac{1}{(\varepsilon |x|)^{d}} \int_{\R^d} A(y) \eta\left(\frac{x-y}{\varepsilon \vert x \vert}\right)\, \ud y.
	$$
	Differentiating this formula, we obtain for all $x \in \overline{\Omega}\setminus\{0\}$,
	\begin{align*}
		x\cdot \nabla A_\varepsilon(x) 
		& = 
		\sum_{i = 1}^d x_i\partial_{x_i} A_\varepsilon(x)
		\\
		& = \frac{1}{(\varepsilon \vert x \vert)^{d}} \int_{\R^d} A(y)\sum_{i,j} \partial_{x_j}\eta \left(\frac{x-y}{\varepsilon \vert x \vert}\right)\bigg[\frac{1}{\varepsilon \vert x \vert}\bigg( x_i \delta_{ij}-(x_j-y_j)\frac{x^2_i}{\vert x\vert^2}\bigg)\bigg] \ud y - d A_\varepsilon(x)
		\\
		& = 
		\frac{1}{(\varepsilon \vert x \vert)^{d+1}} \int_{\R^d} A(y) y \cdot \nabla_{x} \eta \left(\frac{x-y}{\varepsilon \vert x \vert}\right) \ud y - d A_\varepsilon(x).
	 \end{align*}
	We then remark that 
	$$
		\nabla_{x} \eta \left(\frac{x-y}{\varepsilon \vert x \vert}\right) 
		= 
		- \varepsilon |x| \nabla_y \left( \eta \left(\frac{x-y}{\varepsilon \vert x \vert}\right)\right).
	$$	
	We thus get, for all $x \in \overline\Omega \setminus\{0\}$, 
	\begin{align*}
		x \cdot \nabla A_\varepsilon(x) 
		& =
		- \frac{1}{(\varepsilon \vert x \vert)^{d}} \int_{\R^d} A(y) y \cdot \nabla_{y}\left( \eta \left(\frac{x-y}{\varepsilon \vert x \vert}\right) \right) \ud y - d A_\varepsilon(x)
		\\
		& = 
		-  \frac{1}{(\varepsilon \vert x \vert)^{d}} \int_{\R^d} A(y) \div_y\left (y\, \eta \left(\frac{x-y}{\varepsilon \vert x \vert}\right)\right) \ud y.
	\end{align*}
	Using \eqref{ObsCond-a}, whose interpretation is \eqref{Eq-Formulation-Weak-Multiplier}, we thus deduce that for $\varepsilon >0$ small enough, for all $\xi \in \R^d$, for all $x \in \overline\Omega\setminus\{0\}$, 
	$$
		(x \cdot \nabla A_\varepsilon (x) - \beta A_\varepsilon(x)) \xi \cdot \xi \leq 0. 
	$$
	 This condition can also be checked immediately close to $0$ if $0 \in \overline{\Omega}$ by using that in this case $A$ is $\mathscr{C}^1$ in a neighborhood of $0$, which implies that condition \eqref{ObsCond-a} is satisfied pointwise in a neighborhood of $0$ and $\beta A(0) \xi \cdot \xi \geq 0$. 
	\medskip\\	
	{\it Proof of item (v).} Take $\varepsilon_1 \in (0, \varepsilon_0)$ such that 
    %$\varepsilon_1 \sup_{x \in \mathcal{V}_0} |x| \leq d(\partial\mathcal V, \partial \mathcal{V}_0)$. 
    $\varepsilon_1 \sup_{x \in \mathcal{V}_1} |x| \leq d(\mathcal V_1, \partial \mathcal{V})$. 
    Then we can use \eqref{Other-Formula-For-A-eps} to deduce that for all 
    %$x \in \mathcal{V}_0 \setminus \{0\}$
    $x \in \mathcal{V}_1 \setminus \{0\}$, 
    $\varepsilon < \varepsilon_1$ and  $i \in \{1, \cdots, d\}$, we have
	$$
		\partial_i A_\varepsilon(x) =  \int_{ B(0,1)} \left(\partial_i A(x- \varepsilon |x| y) - \varepsilon \frac{x_i}{|x|} \sum_{j = 1}^d \left(y_j \partial_j A(x - \varepsilon |x| y)\right) \right)\eta (y) \, \ud y.
	$$
	Item $(v)$ follows immediately.
\end{proof}

%Positivity \textit{(i)} is preserved by the non-negative property of \( \eta \), the \( C^1 \) regularity $(ii)$ follows from the smoothing properties of the kernel, and the multiplier type condition $(iv)$ is verified pointwise by substituting the parameter \( (2-\alpha) \) with \( \beta \) in \eqref{ObsCondRho}. Regarding the convergence $(iii)$, the boundedness of the domain \( \Omega \) is used. In this setting, convergence in \( L^\infty \) directly implies convergence in \( L^1 \), which guarantees that \( a_\varepsilon \to a \) strongly in \( L^1(\Omega) \).

\subsection{The multiplier method for the wave equation}
\label{Subsec-Multiplier-wave}
The goal of this section is to prove Theorem \ref{MainThmOBS} and Theorem \ref{MainThmOBS-distributed}.

As a first step, we prove uniform observability estimates for the wave equation using multiplier methods for a class of regular main coefficients  $A$ satisfying condition \eqref{ObsCond-a}.
 We then prove Theorem \ref{MainThmOBS} by passing to the limit in the observability estimate.

%Our strategy first consists in proving uniform observability estimates for the wave equation using multiplier methods for a class of regular main coefficients $A$ satisfying condition \eqref{ObsCond-a}. We then show Theorem \ref{MainThmOBS} by passing to the limit in the observability estimate.

%Multiply by $ x \cdot \nabla y + \lambda y$ for a suitable choice of $\lambda >0$ (depending on $\beta$). 

\subsubsection{The case of smooth coefficients}
%%  Inicio P1

We first prove an observability result for the wave equation \eqref{Wave-eq-intro} for smooth coefficients $A \in \mathscr{C}^1$, satisfying the multiplier condition \eqref{ObsCond-a}. Such observability estimate is classical and can be found for instance in \cite{Komornik-1989}. The main point here is that the observability estimates obtained in this way are uniform over  a large class of coefficients $A$, in the sense that they depend only on the ellipticity constant and the $L^\infty$-bound of $A$, on $\beta$ in \eqref{ObsCond-a} and  the geometry. This  is why we present the detailed proof afterwards.

\begin{proposition} \label{ObsC1}
	
	Let $\Omega$ be a non-empty open subset of $\R^d$ of class $\mathscr{C}^2$.

    Let $A \in \mathscr{C}^1(\overline{\Omega} \setminus \{0\}; \R^{d\times d}) \cap W^{1,\infty}(\Omega; \R^{d \times d})$ satisfying \eqref{Ass-Symmetric-A} with constant $\alpha$ and assume that there exists $\beta_1 \in [0, 2)$ such that 
	\begin{equation}
			\label{ObsCond-a-smooth}	
			\forall x \in \overline\Omega,\ \forall \xi \in \R^d, \quad
			((x\cdot \nabla A - \beta_1 A ) \xi) \cdot \xi    \leq 0.
	\end{equation} 
	Let $\Gamma_0$ be defined by \eqref{Multiplier-Set-Gamma}.   

	Then there exist $T>0$ and $ C > 0$, depending only on $\alpha $, $\beta_1$, $T$ and the geometry (i.e. of the domain $\Omega$ and of $\Gamma_0$), such that
    any solution $y$ of equation \eqref{Wave-eq-intro} within initial data in $H^1_0(\Omega) \times L^2(\Omega)$ satisfies \eqref{Obs-Boundary}. To be more precise, for any $T$ satisfying
	\begin{equation}
		\label{Cond-T}
		T > \frac{4 \sqrt{\alpha} \sup_{x \in \Omega} |x|}{2 -  \beta_1}, 
	\end{equation}
	the estimate \eqref{Obs-Boundary} holds with constant $C$ given by 
	\begin{equation}
		\label{Constant-Boundary}
		C^2 = \frac{\alpha}{T ( 1 - \beta_1/2) -2 \sqrt{\alpha} \sup_{x \in \Omega} |x|}.
	\end{equation}
    
    Similarly, if $\omega$ is a non-empty open subset of $\Omega$ such that $\omega$ is a neighborhood of 
    $ \Gamma_0  = \{ x \in \partial \Omega, \, x \cdot \nu >0 \}$, for any $T>0$ satisfying \eqref{Cond-T}, there exists a constant $ C_\omega > 0$, depending only on $\alpha $, $\beta_1$, $T$ and the geometry (i.e. of the domains $\omega$ and $ \Omega$) such that any solution $y$ of equation \eqref{Wave-eq-intro} with initial data $(y_0,y_1) \in H_0^1(\Omega) \times L^2(\Omega)$ satisfies 
	\begin{equation}\label{Obs-Distributed-smooth}
		\int_\Omega \left( A \nabla y_0 \cdot \nabla y_0+ |y_1|^2  \right) \ud x 
	\leq 
		 C^2_\omega \int_0^T\int_{\omega \cap \Omega}\left( \left| \partial_t y\right|^2 +  \left| y\right|^2 \right)
	 	  \, \ud x\, \ud t.  
	\end{equation}
\end{proposition}

\begin{proof}
	Let $y$ be a solution of the wave equation \eqref{Wave-eq-intro} with initial data  $(y_0, y_1) \in H^2 \cap H^1_0(\Omega) \times H^1_0(\Omega)$. The trajectory $y$ then belongs to $\mathscr{C}^2([0,T]; H^2 \cap H^1_0(\Omega) ) \cap \mathscr{C}^1([0,T]; H^1(\Omega))$. Such regularity will allow us to justify all the computations afterwards.
	
	We multiply the wave equation \eqref{Wave-eq-intro} by $x \cdot \nabla y$. This yields
	\begin{align*}
		0 &= \int_0^T \int_\Omega (\partial_{tt} y - \div(A \nabla y)) x\cdot \nabla y \, \ud x\, \ud t
		\\
		& = 
		\int_\Omega \partial_t y \, x \cdot \nabla y \, \ud x\Big|_0^T 
		- 
		\int_0^T \int_\Omega x \cdot \nabla \left( \frac{1}{2} |\partial_t y|^2 \right) \ud x\, \ud t
		\\
		& \qquad - 
		\int_0^T \int_{\partial \Omega} (A \nabla y \cdot \nu) (x \cdot \nabla y) \, \ud\sigma \, \ud t
		+ 
		\int_0^T \int_\Omega A \nabla y \cdot \nabla ( x \cdot \nabla y) \, \ud x \ud t.
		\\ 
		& = 
		\int_\Omega \partial_t y \, x \cdot \nabla y \, \ud x\Big|_0^T 
		+ 
		\frac{d}{2}  
		\int_0^T \int_\Omega |\partial_t y|^2  \ud x\, \ud t
		\\
		& \qquad - 
		\int_0^T \int_{\partial \Omega} |\partial_\nu y|^2 (A \nu \cdot \nu ) (x \cdot \nu) \, \ud\sigma \, \ud t
		+ 
		\int_0^T \int_\Omega A \nabla y \cdot \nabla ( x \cdot \nabla y) \, \ud x \ud t, 
	\end{align*}
	where we used that, since  $y = 0$ on $(0,T) \times \partial \Omega$,
    we have  $\nabla y = \partial_\nu y\ \nu$ and thus $(A \nabla y \cdot \nu) ( x \cdot \nabla y) =| \partial_\nu y|^2 (A \nu \cdot \nu ) (x \cdot \nu)$ on the boundary. 
	
	Now, since $A$ is symmetric, 
	$$
		\int_\Omega A \nabla y \cdot \nabla ( x \cdot \nabla y) \, \ud x 
		= 
		\int_\Omega A \nabla ( x \cdot \nabla y)  \cdot \nabla y \, \ud x.
	$$
	Hence, 
	\begin{align*}
		2 \int_\Omega A \nabla y \cdot \nabla ( x \cdot \nabla y) \, \ud x 
		& 
		=
		\sum_{i,j,k = 1}^d \int_\Omega a_{i,j} \left( (\partial_i y) \partial_j( x_k \partial_k y ) + (\partial_i (x_k \partial_k y) ) (\partial_j y)\right) \, \ud x 
		\\
		& 
		= 
		\sum_{i,j,k = 1}^d \int_\Omega a_{i,j} x_k \partial_k ( (\partial_i y )(\partial_j y )) \, \ud x 
		+ 
		2 \int_\Omega A  \nabla y \cdot \nabla y \, \ud x  
		\\
		& 
		= 
		\int_\Omega x \cdot \nabla_x (  A  \nabla y \cdot \nabla y) \, \ud x 
		- 
		\int_\Omega ((x \cdot \nabla_x A)  \nabla y) \cdot \nabla y) \, \ud x
		+ 
		2 \int_\Omega A  \nabla y \cdot \nabla y \, \ud x  
		\\ 
		& = 
		\int_{\partial \Omega} (A \nabla y \cdot \nabla y)  x\cdot \nu \, \ud \sigma
		- 
		(d-2)
		\int_\Omega A  \nabla y \cdot \nabla y \, \ud x  
		- 
		\int_\Omega ((x \cdot \nabla_x A)  \nabla y) \cdot \nabla y) \, \ud x.
	\end{align*}
	Accordingly, we have the identity
	\begin{multline}
		\label{Multiplier-Identity-00-Wave}
		0 
		= 
		\int_\Omega \partial_t y \, x \cdot \nabla y \, \ud x\Big|_0^T 
		+ 
		\frac{d}{2}  
		\int_0^T \int_\Omega |\partial_t y|^2  \ud x\, \ud t
		 - \frac{1}{2} 
		\int_0^T \int_{\partial \Omega} |\partial_\nu y|^2 (A \nu \cdot \nu ) (x \cdot \nu) \, \ud\sigma \, \ud t
		\\
		- 
		\frac{(d-2)}{2} 
		\int_0^T \int_\Omega A  \nabla y \cdot \nabla y \, \ud x \, \ud t
		- 
		\frac{1}{2}\int_0^T \int_\Omega ((x \cdot \nabla_x A)  \nabla y) \cdot \nabla y) \, \ud x \, \ud t.
	\end{multline}
	
	Multiplying the equation \eqref{Wave-eq-intro} by $y$ and integrating in space and time, we also get
	\begin{equation}
		\label{Equipartition-Energy-Wave}
		0 =
		\int_\Omega \partial_t y y \, \ud x\Big|_0^T 
		-
		\int_0^T \int_\Omega |\partial_t y|^2 \, \ud x\,  \ud t
		+
		\int_0^T \int_\Omega A  \nabla y \cdot \nabla y \, \ud x \, \ud t.
	\end{equation}
	Given a real number $\lambda$,  which  will be chosen later, we 
    multiply  identity \eqref{Equipartition-Energy-Wave} by  $\lambda$ and add it to  identity \eqref{Multiplier-Identity-00-Wave}. We obtain
	\begin{multline}
		\label{Multiplier-Identity-0-Wave}
		0 
		= 
		\int_\Omega \partial_t y \, (x \cdot \nabla y + \lambda y ) \, \ud x\Big|_0^T 
		+ 
		\left(\frac{d}{2}  - \lambda\right)
		\int_0^T \int_\Omega |\partial_t y|^2  \ud x\, \ud t
		 - \frac{1}{2} 
		\int_0^T \int_{\partial \Omega} |\partial_\nu y|^2 (A \nu \cdot \nu ) (x \cdot \nu) \, \ud\sigma \, \ud t
		\\
		+ \left(\lambda- 
		\frac{(d-2)}{2} \right)
		\int_0^T \int_\Omega A  \nabla y \cdot \nabla y \, \ud x \, \ud t
		- 
		\frac{1}{2}\int_0^T \int_\Omega ((x \cdot \nabla_x A)  \nabla y) \cdot \nabla y) \, \ud x \, \ud t.
	\end{multline}
	Taking 
	\begin{equation}
		\label{Choice-Lambda}
		\lambda = \frac{d-1}{2}  + \frac{\beta_1}{4}, 
	\end{equation}
	and using \eqref{ObsCond-a-smooth}, we get 
	\begin{multline*}
		0 \geq  
		\int_\Omega \partial_t y \, (x \cdot \nabla y + \lambda y ) \, \ud x\Big|_0^T 
		+ 
		\left(\frac{1}{2}  - \frac{\beta_1}{4} \right)
		\left( \int_0^T \int_\Omega |\partial_t y|^2  \ud x\, \ud t + \int_0^T \int_\Omega A  \nabla y \cdot \nabla y \, \ud x \, \ud t \right)
		\\
		 - \frac{1}{2} 
		\int_0^T \int_{\partial \Omega} |\partial_\nu y|^2 (A \nu \cdot \nu ) (x \cdot \nu) \, \ud\sigma \, \ud t.
	\end{multline*}
	Using the bound \eqref{Ass-Symmetric-A}, the multiplier condition \eqref{ObsCond-a-smooth} and the fact that the energy is preserved (recall Theorem \ref{Thm-WP}), we obtain
	\begin{multline}
			\label{Multiplier-Est-1-Wave}
		T \left(\frac{1}{2}  - \frac{\beta_1}{4} \right)
		\left( \int_\Omega\left(A  \nabla y_0 \cdot \nabla y_0 + | y_1|^2 \right) \, \ud x \right)
		\leq 
		 \frac{\alpha}{2} 
		\int_0^T \int_{\Gamma_0} |\partial_\nu y|^2  (x \cdot \nu) \, \ud\sigma \, \ud t
		\\
		 + \left| 	\int_\Omega \partial_t y(T) \, (x \cdot \nabla y(T) + \lambda y(T) ) \, \ud x \right| 
		 + \left| 	\int_\Omega \partial_t y(0) \, (x \cdot \nabla y(0) + \lambda y(0) ) \, \ud x \right| .
	\end{multline}
	Now, remark that for all $z \in H^1_0(\Omega)$, 
	\begin{align}
		\| (x \cdot \nabla z + \lambda z ) \|_{L^2(\Omega)}^2
		& \leq 
		\int_\Omega \left( (x \cdot \nabla z)^2 + \lambda^2 |z|^2 +  \lambda x \cdot \nabla (|z|^2) \right) \ud x
		\notag
		\\
		& \leq  
		\int_\Omega \left( (x \cdot \nabla z)^2 + (\lambda^2 - d \lambda) |z|^2 \right) \ud x
		\notag
		\\
		\label{Est-x-nabla-z}
		& \leq 
		\int_\Omega (x \cdot \nabla z)^2  \ud x = \left(\sup_{x \in \Omega} |x|\right)^2 \|  \nabla z\|_{L^2(\Omega)}^2
		\leq \alpha \left(\sup_{x \in \Omega} |x|\right)^2 \int_\Omega A \nabla z \cdot \nabla z \, \ud x,  
	\end{align}
	where we used that, since $\beta_1 \in [0, 2)$ and $\lambda \in [0,d)$, 
    we have $\lambda^2 - d \lambda \leq 0$.
	
	Accordingly, 
	\begin{align}
		 \left| 	\int_\Omega \partial_t y(0) \, (x \cdot \nabla y(0) + \lambda y(0) ) \, \ud x \right| 
		& \leq 
		\sqrt{\alpha} \left(\sup_{x \in \Omega} |x|\right) \| \partial_t y(0) \|_{L^2(\Omega)} \| (A \nabla y (0)\cdot \nabla y(0))^{1/2} \|_{L^2(\Omega)} 
		 \notag
		\\
		 &
		 \leq 
		\frac{ \sqrt{\alpha}  \left(\sup_{x \in \Omega} |x|\right)}{2}\left( \int_\Omega \left(A  \nabla y_0 \cdot \nabla y_0 + | y_1|^2 \right) \, \ud x \right), 
		\label{Bound-BT-t=0}
	\end{align}
	and, using that the energy is constant, 
	\begin{align}
		 \left| 	\int_\Omega \partial_t y(T) \, (x \cdot \nabla y(T) + \lambda y(T) ) \, \ud x \right| 
		& \leq 
		\frac{ \sqrt{\alpha}  \left(\sup_{x \in \Omega} |x|\right)}{2} 
		\left( \int_\Omega \left( |\partial_t y(T)|^2  +  A  \nabla y(T) \cdot \nabla y(T) \right) \, \ud x \right)
		\notag
		\\
		&
		\leq
		\frac{ \sqrt{\alpha} \left(\sup_{x \in \Omega} |x|\right)}{2} \left( \int_\Omega \left(A  \nabla y_0 \cdot \nabla y_0 + | y_1|^2 \right) \, \ud x \right).
		\label{Bound-BT-t=T}
	\end{align}
	We thus deduce from \eqref{Multiplier-Est-1-Wave} and the above estimates that 
	$$
		\left( T \left(\frac{1}{2}  - \frac{\beta_1}{4}\right)-  \sqrt{\alpha}  \left(\sup_{x \in \Omega} |x|\right)   \right)
		\left( \int_\Omega\left(A  \nabla y_0 \cdot \nabla y_0 + | y_1|^2 \right) \, \ud x \right)
		\leq 
		 \frac{\alpha}{2} 
		\int_0^T \int_{\Gamma_0} |\partial_\nu y|^2  (x \cdot \nu) \, \ud\sigma \, \ud t.
	$$
	Therefore, for any $T$ satisfying \eqref{Cond-T}, 
	the estimate \eqref{Obs-Boundary} holds with the constant $C$ given by \eqref{Constant-Boundary}. This estimate has been proved for any solution $y$ of \eqref{Wave-eq-intro} with initial data in the space $H^2(\Omega) \cap H^1_0(\Omega) \times H^1_0(\Omega)$, which is dense in $H^1_0(\Omega) \times L^2(\Omega)$. Since all terms in the estimate \eqref{Obs-Boundary} depend continuously on the initial data in $H^1_0(\Omega) \times L^2(\Omega)$, the estimate \eqref{Obs-Boundary} holds for any solution $y$ of \eqref{Wave-eq-intro} with initial data in $H^1_0(\Omega) \times L^2(\Omega)$.
	
	To prove observability from a distributed set $\omega$ satisfying the assumptions of Proposition \ref{ObsC1}, we use a suitably cut-off version of the previous computations. Namely, we introduce a smooth ($\mathscr{C}^\infty$) cut-off function $\eta = \eta(x)$ 
    taking values in $[0,1]$, which is
    equal to $1$ in $\overline\Omega \setminus \omega$
    and
    vanishes  in a neighborhood of $\{ x \in \partial \Omega, \, x \cdot \nu >0 \}$. 
    For further use, we also assume that $\Supp(1-\eta) \subset \omega$. Then, instead of multiplying the equation \eqref{Wave-eq-intro} by $x \cdot \nabla y+ \lambda y$ as before, we multiply the equation \eqref{Wave-eq-intro} by $\eta(x) (x \cdot \nabla + \lambda y)$. Integrating in $x$ and $t$ we obtain
	\begin{multline}
		\label{Multiplier-Identity-0-Wave-Dis}
		0 
		= 
		\int_\Omega \eta \partial_t y \, (x \cdot \nabla y + \lambda y ) \, \ud x\Big|_0^T 
		+ 
		\left(\frac{d}{2}  - \lambda\right)
		\int_0^T \int_\Omega \eta |\partial_t y|^2  \ud x\, \ud t
		+ \frac{1}{2} \int_0^T \int_\Omega x \cdot \nabla \eta |\partial_t y|^2  \ud x\, \ud t
		\\
		 - \frac{1}{2} 
		\int_0^T \int_{\partial \Omega} \eta |\partial_\nu y|^2 (A \nu \cdot \nu ) (x \cdot \nu) \, \ud\sigma \, \ud t
		+ \left(\lambda- 
		\frac{(d-2)}{2} \right)
		\int_0^T \int_\Omega \eta A  \nabla y \cdot \nabla y \, \ud x \, \ud t
		\\
		- 
		\frac{1}{2}\int_0^T \int_\Omega \eta ((x \cdot \nabla_x A)  \nabla y) \cdot \nabla y) \, \ud x \, \ud t
		\\
		+ 
		\int_0^T \int_\Omega A  \nabla y \cdot \nabla \eta ( (x \cdot \nabla y + \lambda y )  \, \ud x \, \ud t
		- 
		\frac{1}{2}\int_0^T \int_\Omega x \cdot \nabla \eta ( A  \nabla y \cdot \nabla y) \, \ud x \, \ud t
		.
	\end{multline}
	Choosing $\lambda$ as in \eqref{Choice-Lambda}, using the bound \eqref{Ass-Symmetric-A}, the multiplier condition \eqref{ObsCond-a-smooth}, the bounds \eqref{Bound-BT-t=0}--\eqref{Bound-BT-t=T} and the fact that $\nabla \eta$ is supported in $\omega_0 := \Supp(1 - \eta) \cap \Omega$ and that $\eta = 0$ on $\{x \in \partial\Omega, \, x\cdot \nu > 0\}$, we get:
	\begin{multline*}
		\left(\frac{1}{2}  - \frac{\beta_1}{4}\right) \int_0^T \int_\Omega \eta \left(A  \nabla y \cdot \nabla y + |\partial_t y|^2 \right)\, \ud x \, \ud t
		\leq  \sqrt{\alpha}  \left(\sup_{x \in \Omega} |x|\right) 
		\left( \int_\Omega\left(A  \nabla y_0 \cdot \nabla y_0 + | y_1|^2 \right) \, \ud x \right)
		\\
		+ 
		\| \nabla \eta\|_{L^\infty} \left(\sup_{x \in \Omega} |x|\right) 
		\left( 
			\frac{1}{2} \int_0^T \int_{\omega_0}  |\partial_t y|^2\, \ud x \, \ud t
			+
			\frac{3}{2} \alpha \int_0^T \int_{\omega_0} |\nabla y|^2 \ud x \, \ud t
		\right).
	\end{multline*}
	Adding 
	$$
		\left(\frac{1}{2}  - \frac{\beta_1}{4}\right) \int_0^T \int_\Omega (1-\eta) \left(A  \nabla y \cdot \nabla y + |\partial_t y|^2 \right)\, \ud x \, \ud t
	$$
	on both sides and using that the energy is preserved, we deduce:
	\begin{multline*}
		T \left(\frac{1}{2}  - \frac{\beta_1}{4}\right)  \int_\Omega \left(A  \nabla y_0 \cdot \nabla y_0 + | y_1|^2 \right)\, \ud x
		\leq  \sqrt{\alpha} \left(\sup_{x \in \Omega} |x|\right) 
		\left( \int_\Omega\left(A  \nabla y_0 \cdot \nabla y_0 + | y_1|^2 \right)\, \ud x \right)
		\\
		+ 
		\| \nabla \eta\|_{L^\infty} \left(\sup_{x \in \Omega} |x|\right) 
		\left( 
			\frac{1}{2}  \int_0^T \int_{\omega_0}  |\partial_t y|^2\, \ud x \, \ud t
			+
			\frac{3}{2} \alpha \int_0^T \int_{\omega_0} |\nabla y|^2 \ud x \, \ud t
		\right)
		\\
		+ 
		 \left(\frac{1}{2}  - \frac{\beta_1}{4}\right)\int_0^T \int_{\omega_0} \left( \alpha | \nabla y|^2 + |\partial_t y|^2  \right)\, \ud x\, \ud t.
	\end{multline*}
	Therefore, for any $T$ satisfying \eqref{Cond-T}, there exists a constant $C$ depending on $T$, $\alpha$, $\beta_1$ and the geometry such that any solution $y$ of \eqref{Wave-eq-intro} with initial data in $H^1_0(\Omega) \times L^2(\Omega)$ satisfies
	\begin{equation*}
		\int_\Omega \left(A  \nabla y_0 \cdot \nabla y_0 + | y_1|^2 \right)\, \ud x
		\leq  
		C^2 \int_0^T \int_{\omega_0} \left( | \nabla y|^2 + |\partial_t y|^2 + |y|^2  \right)\, \ud x\, \ud t.
	\end{equation*}
	Let us finally remove the observation term in $|\nabla y|^2$ in $(0,T) \times \omega_0$. For $T$ satisfying condition \eqref{Cond-T}, we choose $T_0$ satisfying \eqref{Cond-T} and $T_0 < T$ and $\epsilon \in (0, T-T_0)$. Using the fact that the energy is preserved, we clearly have the existence of a constant $C_0$ depending on $T$, $\alpha$, $\beta_1$ and the geometry such that any solution $y$ of \eqref{Wave-eq-intro} with initial data in $H^1_0(\Omega) \times L^2(\Omega)$ satisfies
	\begin{equation}
		\label{Distributed-Obs-with-all-the-H1-norm}
		\int_\Omega \left(A  \nabla y_0 \cdot \nabla y_0 + | y_1|^2 \right)\, \ud x
		\leq  
		C^2 \int_{\epsilon}^{T_0 + \epsilon} \int_{\omega_0} \left( | \nabla y|^2 + |\partial_t y|^2 + |y|^2  \right)\, \ud x\, \ud t.
	\end{equation}
	We then take $\eta_0$ a smooth ($\mathscr{C}^\infty_{t,x}$) cut-off function such that $\eta_0 = 1$ in $(\epsilon, T_0 + \epsilon)\times \omega_0$ and compactly supported in $(0,T) \times \omega$ (since $\omega_0 = \Supp(1 - \eta)\cap \Omega$ is a closed set included in $\omega$). Then, multiplying the equation \eqref{Wave-eq-intro} by $\eta_0^2 y$, we get: 
	\begin{multline*}
		0 
		= 
		- 
		\int_0^T \int_\Omega \eta_0^2 |\partial_t y|^2 \, \ud x\, \ud t
		+
		\frac{1}{2}  
		\int_0^T \int_\Omega \partial_{tt} (\eta_0^2) |y|^2 \, \ud x\, \ud t
		\\
		+ 
		\int_0^T \int_\Omega \eta_0^2 A \nabla y \cdot \nabla y  \, \ud x\, \ud t
		+ 
		2 \int_0^T \int_\Omega \eta_0 A \nabla y \cdot \nabla \eta_0 \, y  \, \ud x\, \ud t.		
	\end{multline*}
	Accordingly, for all $\gamma >0$, 
	\begin{align*}
		&\frac{1}{\alpha} \int_0^T \int_\Omega \eta_0^2 |\nabla y |^2  \, \ud x\, \ud t
		 \leq 
		\int_0^T \int_\Omega \eta_0^2 A \nabla y \cdot \nabla y  \, \ud x\, \ud t
		\\
		&
		\leq
		\int_0^T \int_\Omega \eta_0^2 |\partial_t y|^2 \, \ud x\, \ud t
		+
		\frac{1}{2}  \| \partial_{tt} (\eta_0^2)\|_{L^\infty}
		  \int_0^T \int_{\omega\cap \Omega} |y|^2 \, \ud x\, \ud t
		+
		\alpha \int_0^T \int_\Omega\left( \frac{1}{\gamma} \eta_0^2 |\nabla y|^2 + \gamma |\nabla \eta_0|^2 |y|^2\right)\, \ud x\, \ud t.
	\end{align*}
	Choosing $\gamma = 2 \alpha^2$, we deduce 
	\begin{align*}
		\frac{1}{2\alpha} \int_0^T \int_\Omega \eta_0^2 |\nabla y |^2  \, \ud x\, \ud t
		\leq
		\int_0^T \int_\Omega \eta_0^2 |\partial_t y|^2 \, \ud x\, \ud t
		+
		\left( \frac{1}{2}  \| \partial_{tt} (\eta_0^2)\|_{L^\infty} + 2 \alpha^3 \| \nabla \eta_0\|^2_{L^\infty} \right)
		\int_0^T \int_{\omega\cap \Omega} |y|^2 \, \ud x\, \ud t.
	\end{align*}
	Since $\eta_0 = 1$ in $(\epsilon, T_0+\epsilon) \times \omega_0$, we deduce the observability property \eqref{Obs-Distributed-smooth} immediately from \eqref{Distributed-Obs-with-all-the-H1-norm} and the previous estimate.
\end{proof}

\subsubsection{The case of coefficients in $L^\infty$: Proof of Theorem \ref{MainThmOBS} and Theorem \ref{MainThmOBS-distributed}}

Applying the approximation arguments we have developed, it is now easy to deduce Theorems \ref{MainThmOBS} and \ref{MainThmOBS-distributed}.

\begin{proof}[Proof of Theorem \ref{MainThmOBS}.]
	Let us consider $A\in L^\infty(\Omega; \R^{d\times d})$ satisfying the assumptions \ref{Assumption1} and such that $ A$ is in $W^{1,\infty}$ in a neighborhood $\mathcal{V}$ of $\{ x\in \partial\Omega, \, x \cdot \nu >0\}$.
	
	Then, for $n \in \N$ sufficiently large, we define $A_n$ as the one given by Proposition \ref{Prop-Approx-a} for $\varepsilon = 1/2^n$. 

	Using Proposition \ref{Prop-Approx-a}, since all these matrices $(A_n)_{ n \geq n_0}$ satisfy the assumptions of Proposition \ref{ObsC1} with $\beta_1=\beta$, and $\sup_{n \geq n_0} \| \nabla A_n\|_{L^\infty( \mathcal{V})} < \infty$. Accordingly, we obtain uniform (in $n \geq n_0$) obser\-vability inequalities for the equations \eqref{Wave-eq-a-n}. More precisely, for all $T$ satisfying \eqref{Cond-T}, there exists a constant $C$ such that for all $n \geq n_0$, any solution $y_n$ of \eqref{Wave-eq-a-n} with initial data in $H^1_0(\Omega) \times L^2(\Omega)$ satisfies
	$$
		\int_\Omega \left(A_n \nabla y_0 \cdot \nabla y_0 + |y_1|^2 \right) \, \ud x
		\leq
		C^2 
		\int_0^T \int_{\Gamma_0} |\partial_\nu y_n|^2\, \ud \sigma \, \ud t.  
	$$
	
	We can then pass to the limit $n \to \infty$ in the above estimates, using Theorem \ref{Thm-Regularization} and Corollary \ref{Cor-Conv-Normal-Trace}, and we get the observability inequality \eqref{Obs-Boundary} for any solution $y$ of \eqref{Wave-eq-intro} with initial data in $H^1_0(\Omega) \times L^2(\Omega)$. 
\end{proof}

\begin{proof}[Proof of Theorem \ref{MainThmOBS-distributed}.]
	Arguing as in the proof of Theorem \ref{MainThmOBS}, we easily deduce that for all $T$ satisfying \eqref{Cond-T}, there exists a constant $C$ such that any solution $y$ of of \eqref{Wave-eq-intro} with initial data in $H^1_0(\Omega) \times L^2(\Omega)$ satisfies:
	\begin{equation}
		\label{Obs-Distributed-With-A-Compact-Term}
		\int_\Omega \left(A \nabla y_0 \cdot \nabla y_0 + |y_1|^2 \right) \, \ud x
		\leq
		C_\omega^2 
		\int_0^T \int_{\omega\cap \Omega} \left( |\partial_t y|^2 + |y|^2\right)\, \ud x \, \ud t.  
	\end{equation}
	Now, to get the observability estimate \eqref{Obs-Distributed}, we should remove the compact term 	$$
		\int_0^T \int_{\omega\cap \Omega} |y|^2\, \ud x \, \ud t.
	$$  
	In order to do that, we follow the classical compactness-uniqueness method and apply the strategy of \cite[Section 3.2]{Dehman-Erv-2016}. 
	
	Let $T >0$ satisfy \eqref{Cond-T} and assume that the observability estimate \eqref{Obs-Distributed} is false, i.e. that there exists a sequence of initial data $(y_{0,n}, y_{1,n}) \in H^1_0(\Omega) \times L^2(\Omega)$ such that 
	\begin{equation}
		\label{ByContradiction}
		\forall n \in \N, \quad \int_\Omega \left(A \nabla y_{0,n} \cdot \nabla y_{0,n} + |y_{1,n}|^2 \right) \, \ud x = 1
		\quad
		\text{ and } 
		\quad 
		\lim_{n \to \infty} \int_0^T \int_{\omega\cap \Omega} |\partial_t y_n|^2 \, \ud x \, \ud t = 0, 
	\end{equation}
	where $y_n$ denotes the solution of \eqref{Wave-eq-intro} with initial datum $(y_{0,n}, y_{1,n})$. 

	Then we have the weak convergence of $((y_{0,n}, y_{1,n}))_{n \in \N}$ to some $(y_0,y_1)$ in $H^1_0(\Omega) \times L^2(\Omega)$ (up to a subsequence still denoted the same), and of the solution $(y_n)_{n \in \N}$ of \eqref{Wave-eq-intro} to the solution $y$ of \eqref{Wave-eq-intro} in $H^1((0,T) \times \Omega)$ with initial datum in $H^1_0(\Omega) \times L^2(\Omega)$. Besides, we have by assumption that $\partial_t y = 0$ in $(0,T) \times (\omega\cap \Omega)$. 
	
	Now, we choose $T_0$ satisfying \eqref{Cond-T} and $T_0 < T$. Then, for $\tau \in (0, T - T_0)$, we define the function $z_{\tau}$ by 
	$$
		z_{\tau} (t, x) = \frac{y(t+\tau, x) - y(t,x)}{\tau} , \qquad  ((t,x) \in (0,T) \times \Omega). 
	$$
	For all $\tau \in (0,T-T_0)$, $z_\tau$ solves the equation \eqref{Wave-eq-intro} in $(0,T_0) \times \Omega$, belongs to $\mathscr{C}^0([0,T_0]; H^1_0(\Omega)) \cap \mathscr{C}^1([0,T_0]; L^2(\Omega))$, and,  since $\partial_t y = 0$ in $(0,T) \times (\omega\cap \Omega)$, $z_\tau = 0$ in $(0,T_0) \times (\omega\cap \Omega)$. We can thus apply the observability estimate \eqref{Obs-Distributed-With-A-Compact-Term} in time $T_0$ to $z_\tau$ for any $\tau \in (0,T-T_0)$. We thus deduce that $z_\tau = 0$ in $(0,T_0) \times \Omega$ for any $\tau \in (0,T-T_0)$. Passing to the limit $\tau \to 0$, we deduce that $\partial_t y = 0$ in $(0,T_0) \times \Omega$, i.e. $y(t,x) = y_0(x)$ for all $(t,x) \in [0,T_0]\times \Omega$. From the equation \eqref{Wave-eq-intro}, we deduce that $y_0$ satisfies
	$$
		- \div (A \nabla y_0) = 0 \text { in } \Omega. 
	$$
	But $y_0$ belongs to $H^1_0(\Omega)$, so that multiplying this equation by $y_0$, we see that $y_0 = 0$. Thus, for all $T_0$ satisfying \eqref{Cond-T} and $T_0 < T$, $y$ vanishes identically on $(0,T_0) \times \Omega$. It follows that $y$ vanishes identically on $(0,T) \times \Omega$. 
	
	Now, the sequence $(y_n)_{n \in \N}$ weakly converges to $y= 0$ in $H^1((0,T) \times \Omega)$, hence, by compactness, the sequence $(y_n)_{n \in \N}$ strongly converges to $y= 0$ in $L^2((0,T) \times \Omega)$. The sequence in \eqref{ByContradiction} would thus contradict the estimate \eqref{Obs-Distributed-With-A-Compact-Term}. 
\end{proof}
%%%%%%%%%
\subsection{Observability properties of the plate equation}
\label{Subsec-Obs-Plates}

In this section, we adapt the strategy developed for the waves to the case of the plate equation \eqref{Plate-Eq} with coefficient $a \in L^\infty(\Omega)$.

In order to do so, as in the case of the wave equation, we start with the case of smooth coefficients $a \in \mathscr{C}^1$ satisfying a multiplier condition, proving uniform observability properties, possibly with some compact loss, for a large class of coefficients. We will then pass to the limit in the observability properties, and finally absorb the compact losses at the end on the limiting equation.

%Before going further let us recall that the plate equations \eqref{Plate-Eq} with coefficient $a \in L^\infty(\Omega; \R)$ bounded from below by a positive constant is well-posed. Solutions $y$ of \eqref{Plate-Eq} with initial data in $H^2_0(\Omega) \times L^2(\Omega)$ belong to $\mathscr{C}^0([0,T]; H^2_0(\Omega))\cap \mathscr{C}^1([0,T]; L^2(\Omega))$. Furthermore the energy of the solutions $y$ of \eqref{Plate-Eq}, defined for $t \in [0,T]$ by 
%\begin{equation}
%    E(t) = \frac{1}{2} \int_{\Omega} \left( a(x)|\Delta y(t,x)|^2+ |\partial_t y(t,x)|^2  \right) \ud\,x
%\end{equation}
%does not depend on $t$.
%

\subsubsection{Observability type estimates for smooth coefficients}

We start with the following result, giving observability estimates with compact loss for smooth coefficients satisfying multiplier type conditions:
\begin{proposition}
	\label{Prop-Plate-Multiplier-Smooth}
	Let $\Omega$ be a non-empty open subset of $\R^d$ of class $\mathscr{C}^2$. Let $a \in \mathscr{C}^1(\overline\Omega; \R)$ satisfying, for some $\alpha >0$, $1/\alpha \leq a \leq \alpha $ in $\Omega$ and, for some $\beta \in [0,4) $, the condition $x\cdot  \nabla a - \beta a \leq 0$ in $\Omega$. 
	
	Let $\Gamma_0$ be defined as in \eqref{Multiplier-Set-Gamma}. Then for any $T>0$ there exists $C_{\Gamma_0}>0$ depending only on $\alpha $, $\beta$ and the geometry (i.e. of the domain $\Omega$ and $\Gamma_0$) such that any solution $y$ of \eqref{Plate-Eq} with initial datum $(y_0, y_1) \in H^2_0(\Omega) \times L^2(\Omega)$  satisfies
	\begin{equation}
		\label{Obs-Plates-Boundary-smooth}
		\int_\Omega \left(a |\Delta y_0|^2 + |y_1|^2\right) \, \ud x 
		\leq 
		C_{\Gamma_0}^2 \int_0^T \int_{\Gamma_0} a |\Delta y|^2\, \ud \sigma\, \ud t
		 + 
		C_{\Gamma_0}^2 \left(\|\nabla y_0\|_{L^2(\Omega)}^2+ \|\nabla y(T)\|_{L^2(\Omega)}^2\right).
	\end{equation}
	Moreover, for $T$ large enough depending only on $\alpha$, $\beta$ and the geometry, there exists $C_{obs, \Gamma}$ depending only on $\alpha $, $\beta$ and the geometry such that any solution $y$ of \eqref{Plate-Eq} with initial datum $(y_0, y_1) \in H^2_0(\Omega) \times L^2(\Omega)$  satisfies
	\begin{equation}
		\label{Obs-Plates-Boundary-smooth-large-time}
		\int_\Omega \left(a |\Delta y_0|^2 + |y_1|^2\right) \, \ud x 
		\leq 
		C_{\Gamma_0}^2 \int_0^T \int_{\Gamma_0} a |\Delta y|^2\, \ud \sigma\, \ud t.
	\end{equation}
	
	Similarly, if $\omega$ is a non-empty open subset of $\Omega$ such that $\omega$ is a neighborhood of $\{ x \in \partial \Omega, \, x \cdot \nu >0 \}$, for any $T>0$ and $T_0, T_1$ such that  $0 < T_0 < T_1 < T$, there exists a constant $ C_\omega > 0$, depending only on $\alpha $, $\beta$, $T_0$, $T_1$, $T$ and the geometry (i.e. of the domains $\omega$ and $ \Omega$) such that any solution $y$ of equation \eqref{Wave-eq-intro} with initial data $(y_0,y_1) \in H_0^2(\Omega) \times L^2(\Omega)$ satisfies 
	\begin{multline}\label{Obs-Plates-Distributed-smooth}
		\int_\Omega \left(a |\Delta y_0|^2 + |y_1|^2\right) \, \ud x 
	\leq 
		 C^2_\omega \int_0^T\int_{\omega\cap \Omega}\left( \left| \partial_t y\right|^2 +  \left| \nabla y\right|^2 +  \left| y\right|^2 \right)
	 	  \, \ud x\, \ud t
	\\
		 + 
		 C_\omega^2 \left(\|\nabla y(T_0)\|_{L^2(\Omega)}^2+ \|\nabla y(T_1)\|_{L^2(\Omega)}^2\right).  
	\end{multline}
	Furthermore, for $T$ large enough depending only on $\alpha$, $\beta$ and the geometry, there exists a constant $C_{obs,\omega}>0$ depending only on $\alpha $, $\beta$ and the geometry such that any solution $y$ of \eqref{Plate-Eq} with initial datum $(y_0, y_1) \in H^2_0(\Omega) \times L^2(\Omega)$  satisfies
	\begin{equation}
		\label{Obs-Plates-Dis-smooth-large-time}
		\int_\Omega \left(a |\Delta y_0|^2 + |y_1|^2\right) \, \ud x 
		\leq 
		 C^2_{obs,\omega} \int_0^T\int_{\omega\cap \Omega}\left( \left| \partial_t y\right|^2 +  \left| \nabla y\right|^2 + +  \left| y\right|^2 \right)
	 	  \, \ud x\, \ud t.
	\end{equation}

\end{proposition}

 \begin{proof}
 	Let $y$ of \eqref{Plate-Eq} be a solution of \eqref{Plate-Eq} with initial data in $H^2_0(\Omega) \times L^2(\Omega)$, and multiply the equation by $x \cdot \nabla y + \lambda y$, for $\lambda \in \R$ to be determined, and do integration by parts\footnote{To make the computations fully rigorous, we should start with solutions $y$ of \eqref{Plate-Eq} with initial data in $H^4 \cap H^2_0(\Omega) \times H^2_0(\Omega)$ and then use the density of $H^4 \cap H^2_0(\Omega) \times H^2_0(\Omega)$ in $H^2_0(\Omega) \times L^2(\Omega)$ and the continuity of each terms in \eqref{Obs-Plates-Boundary-smooth} and \eqref{Obs-Plates-Distributed-smooth}, similarly as for the waves: we skip the details to avoid repetition.}:
	\begin{multline}\label{Plate-Mult-Identity}
	    \int_{\Omega} \partial_t y (x \cdot \nabla y + \lambda y) \Big|_0^T \ud x + \left( \frac{d}{2} - \lambda \right) \int_{0}^{T} \int_{\Omega} |\partial_t y|^2 \dt\, \dx 
	    \\
	    + \int_{0}^{T} \int_{\Omega} |\Delta y|^2 \left[ a \left( 2 + \lambda - \frac{d}{2} \right) - \frac{1}{2} x \cdot \nabla a \right] \dt\, \dx = \frac{1}{2} \int_{0}^{T} \int_{\partial \Omega} a(x \cdot \nu) |\Delta y|^2 \ud \sigma \dt.
	\end{multline}
	 Let us remark that during the computations we used that the clamped conditions $y = \partial_\nu y = 0$ on the boundary imply  $D^2 y\ n = \Delta y \cdot n$, leading to the integral of the boundary  in the right hand side of \eqref{Plate-Mult-Identity}. 

	Using that $x \cdot \nabla a \leq \beta a$ in $\Omega$ with $\beta \in [0,4)$, we choose
	$$
		\lambda = \frac{d}{2} - 1+ \frac{\beta}{4}, 
	$$
	so that estimate \eqref{Plate-Mult-Identity} gives:
	\begin{equation}\label{Plate-Mult-Est}
	    \int_{\Omega} \partial_t y (x \cdot \nabla y + \lambda y)\ud x  \Big|_0^T 
	    \\
	    +  T \left( 1 - \frac{\beta}{4} \right) \int_{\Omega} \left(a|\Delta y_0|^2 + |y_1|^2\right) \dx \leq \frac{1}{2} \int_{0}^{T} \int_{\partial \Omega} a(x \cdot \nu) |\Delta y|^2 \ud \sigma \dt.
	\end{equation}
	Finally, note that, using Poincaré's estimate and the fact that $\Omega$ is bounded, there exists $C_0$ (depending on $\Omega$ and $\beta$) such that for all $z \in H^1_0(\Omega)$, $\| x \cdot \nabla z +\lambda z\|_{L^2(\Omega)} \leq \| x \cdot \nabla z\|_{L^2(\Omega)} + |\lambda |\| z\|_{L^2(\Omega)} \leq C_0 \| \nabla z\|_{L^2(\Omega)}$. Hence 
	\begin{align}
		&
		\left|  
			\int_{\Omega} \partial_t y (x \cdot \nabla y + \lambda y) \ud x  \Big|_0^T
		\right|
		\leq 
		C_0\sup_{t \in \{0,T\}}\{ \| \partial_t y(t)\|_{L^2} \}
		\left(\|\nabla y_0\|_{L^2(\Omega)}+ \|\nabla y(T)\|_{L^2(\Omega)}\right)
		\notag \\
		&\quad \leq 
		C_0 \left( \int_{\Omega} \left(a|\Delta y_0|^2 + |y_1|^2\right) \dx \right)^{1/2} \left(\|\nabla y_0\|_{L^2(\Omega)}+ \|\nabla y(T)\|_{L^2(\Omega)}\right)
		\notag
		\\ 
		& \quad \leq
		\frac{T}{2} \left( 1 - \frac{\beta}{4} \right) \int_{\Omega} \left(a|\Delta y_0|^2 + |y_1|^2\right) \dx
		+ 
		\frac{C_0^2 }{2T(1 -  \beta/4)} \left(\|\nabla y_0\|_{L^2(\Omega)}+ \|\nabla y(T)\|_{L^2(\Omega)}\right)^2
		\label{Bound-Boundary-Plate}.
	\end{align}
	Estimate \eqref{Obs-Plates-Boundary-smooth} follows immediately from \eqref{Plate-Mult-Est} and the previous estimate.
	
	To derive estimate \eqref{Obs-Plates-Boundary-smooth-large-time}, we use the following Poincaré type estimate: there exists $C_P>0$, depending on $\Omega$, such that for all $z \in H^2_0(\Omega)$, $\|\nabla z \|_{L^2(\Omega)} \leq C \| \Delta z\|_{L^2(\Omega)}$, so that 
	\begin{align}
		\left|  
			\int_{\Omega} \partial_t y (x \cdot \nabla y + \lambda y) \ud x  \Big|_0^T
		\right|
		& \leq 
		C_0 C_P \left( \int_{\Omega} \left(a|\Delta y_0|^2 + |y_1|^2\right) \dx \right)^{1/2} \left(\|\Delta y_0\|_{L^2(\Omega)}+ \|\Delta y(T)\|_{L^2(\Omega)}\right)
		\notag
		\\ 
		 & \leq
		4 C_0 C_P \sqrt{\alpha}  \left( \int_{\Omega} \left(a|\Delta y_0|^2 + |y_1|^2\right) \dx \right).
		\label{Bound-For-Large-Time-Using-Poincare}
	\end{align}
	Accordingly, if $T$ is large enough to get $T( 1- \beta/4) > 4 C_0 C_P$, we get \eqref{Obs-Plates-Boundary-smooth-large-time} by combining this previous estimate with \eqref{Plate-Mult-Est}.
	
	Finally, to derive the estimate \eqref{Obs-Plates-Distributed-smooth} with $\omega$ a neighborhood of $\{x \in \partial \Omega, \, x \cdot \nu >0 \}$, as for the wave equation, we multiply the equation \eqref{Plate-Eq} by $\eta (x \cdot \nabla y + \lambda y)$ for $\lambda$ as above and $\eta$ a smooth cut-off function taking value $1$ in $\overline\Omega \setminus \omega$ and taking value in $[0,1]$: Calling $\omega_0 = \Supp ( 1- \eta)\cap \Omega$, that we can suppose to be such that $\overline\omega_0 \subset \omega$, this yields
	\begin{multline*}
		\int_{\Omega} \eta \partial_t y (x \cdot \nabla y + \lambda y) \ud x  \Big|_0^T
		+ 
		\left( 1 - \frac{\beta}{4} \right) \int_{0}^{T} \int_{\Omega} \eta (a|\Delta y|^2 + |\partial_t y|^2) \, \ud x\, \ud t 
		\\
		\leq C \int_0^T \int_{\omega_0} \left(|\partial_t y|^2 + |D^2 y|^2 + |\nabla y|^2 + |y|^2 \right)\, \ud x\, \ud t. 
	\end{multline*}
	We can then add the term 
	$$
	\left( 1 - \frac{\beta}{4} \right) \int_0^T \int_\Omega (1- \eta) (a |\Delta y|^2 + |\partial_t y|^2) \, \ud x \, \ud t.
	$$
	on both sides to get
	\begin{multline*}
		\int_{\Omega} \eta \partial_t y (x \cdot \nabla y + \lambda y) \ud x  \Big|_0^T
		+ 
		T \left( 1 - \frac{\beta}{4} \right)  \int_{\Omega} (a|\Delta y_0|^2 + |y_1|^2) \, \ud x\, \ud t 
		\\
		\leq C \int_0^T \int_{\omega_0} \left(|\partial_t y|^2 + |D^2 y|^2 + |\nabla y|^2 + |y|^2 \right)\, \ud x\, \ud t. 
	\end{multline*}
	For sufficiently large $T$, using \eqref{Bound-For-Large-Time-Using-Poincare}, we obtain the existence of a constant $C$ (with both $T$ and $C$ depending only on $\alpha$, $\beta$ and the geometry) such that any solution $y$ of \eqref{Plate-Eq} satisfies
	\begin{equation}
		\label{Est-Large-Time-Plate}
		\int_{\Omega} (a|\Delta y_0|^2 + |y_1|^2) \, \ud x\, \ud t 
		\\
		\leq C \int_0^T \int_{\omega_0} \left(|\partial_t y|^2 + |D^2 y|^2 + |\nabla y|^2 + |y|^2 \right)\, \ud x\, \ud t. 
	\end{equation}

	If we use the bound \eqref{Bound-Boundary-Plate} instead, we deduce that for all $T>0$, there exists a constant $C>0$ such that for any solution $y$ of \eqref{Plate-Eq},
	\begin{multline*}
		 \int_{\Omega} (a|\Delta y_0|^2 + |y_1|^2) \, \ud x
		 \leq 
		  C \int_0^T \int_{\omega_0} \left(|\partial_t y|^2 + |D^2 y|^2 + |\nabla y|^2 + |y|^2 \right)\, \ud x\, \ud t + C\left( \| \nabla y_0 \|_{L^2(\Omega)}^2 + \| \nabla y_1\|_{L^2(\Omega)}^2 \right). 
	\end{multline*}
	Now, let $T_0$ and $T_1$ be such that   $0 < T_0<T_1 <T$. By time invariance and preservation of the energy, the previous estimate (applied on $(T_0, T_1)$) yield the existence of a constant $C>0$ such that for any solution $y$ of \eqref{Plate-Eq},
	\begin{multline}
		\label{Mult-Est-Plate-Compact-terms-to-estimate}
		 \int_{\Omega} (a|\Delta y_0|^2 + |y_1|^2) \, \ud x
		\\
		 \leq 
		  C \int_{T_0}^{T_1}  \int_{\omega_0} \left(|\partial_t y|^2 + |D^2 y|^2 + |\nabla y|^2 + |y|^2 \right)\, \ud x\, \ud t + C\left( \| \nabla y(T_0) \|_{L^2(\Omega)}^2 + \| \nabla y(T_1) \|_{L^2(\Omega)}^2 \right). 
	\end{multline}
	 Taking $\eta_0$ a smooth cut-off function taking value $1$ in $(T_0,T_1) \times \omega_0$ and compactly supported in $(0,T) \times \omega$, multiplying the equation \eqref{Plate-Eq} by $\eta_0^2 y$ we easily get
	$$
		 \int_{0}^{T} \int_{\omega\cap \Omega} \eta_0^2 |\Delta y |^2 \, \ud x\, \ud t
		 \leq 
		 C \int_0^T \int_{\omega \cap \Omega} \left(|\partial_t y|^2 + |\nabla y|^2 + |y|^2 \right)\, \ud x\, \ud t.
	$$
	Now, writing $|\Delta (\eta_0 y) | \leq \eta_0 |\Delta y| + C (|\nabla y| + |y|)$ in $(0,T)\times \omega$ and the boundary conditions satisfied by $\eta_0 y$, elliptic regularity gives
	$$
		 \int_{0}^{T} \int_{\omega\cap \Omega}  |D^2(\eta_0 y) |^2 \, \ud x\, \ud t
		 \leq
		 C \int_{0}^{T} \int_{\omega\cap \Omega}  |\Delta(\eta_0 y) |^2 \, \ud x\, \ud t
		 \leq 
		 C \int_0^T \int_{\omega\cap \Omega} \left(|\partial_t y|^2 + |\nabla y|^2 + |y|^2 \right)\, \ud x\, \ud t.
	$$
	Combining this estimate with \eqref{Mult-Est-Plate-Compact-terms-to-estimate}, we easily deduce estimate \eqref{Obs-Plates-Distributed-smooth}.
	
	To get estimate \eqref{Obs-Plates-Dis-smooth-large-time}, we just combine this last estimate with \eqref{Est-Large-Time-Plate}.
 \end{proof}
 
 \subsubsection{Observability type estimates for rough coefficients as in Theorem \ref{MainThmOBS-Plates}}
 
 As a consequence of Theorem \ref{Thm-Regularization-Plate} and the approximation result in Proposition \ref{Prop-Approx-a} (see also Remark \ref{Rk-Approx-For-Plates}), we easily deduce the following result for coefficients $a \in L^\infty(\Omega)$ satisfying, in particular, the multiplier condition \eqref{Multiplier-Condition-Plate} (the detailed proof is left to the reader as this is obvious):
 \begin{proposition}
 	\label{Prop-Obs-Up-To-Comp}
 	Under the assumption of Theorem \ref{MainThmOBS-Plates}, the observability type estimates \eqref{Obs-Plates-Distributed-smooth} and \eqref{Obs-Plates-Dis-smooth-large-time} hold true, and the observability type estimate \eqref{Obs-Plates-Boundary-smooth}, \eqref{Obs-Plates-Boundary-smooth-large-time} hold true as well if we further assume that $a $ is Lipschitz in a neighborhood of $\Gamma_0$. 
 \end{proposition}

 \subsubsection{A uniqueness result for plates with rough coefficients} 
 
In order to deduce Theorem \ref{MainThmOBS-Plates}, we will need a uniqueness result:
\begin{proposition}
	\label{Prop-Uniqueness-Plate}
	Under the assumption of Theorem \ref{MainThmOBS-Plates}, if $y$ is a solution of \eqref{Plate-Eq} with initial datum in $H^2_0(\Omega) \times L^2(\Omega)$ satisfying  either $y= 0$ in $(0,T) \times (\omega\cap \Omega)$, or $\Delta y = 0$ in $(0,T) \times \Gamma_0$ and $a$ Lipschitz in a neighborhood of $\Gamma_0$, then $y$ vanishes everywhere in $(0,T) \times \Omega$. 
\end{proposition}

\begin{proof}
	Before going into the proof, it is interesting to point out that solutions $y$ of \eqref{Plate-Eq} with initial data in $H^2_0(\Omega) \times L^2(\Omega)$ not only have constant $H^2_0(\Omega) \times L^2(\Omega)$ energy (recall the definition of the energy in \eqref{Energy-Plate}), but also a weaker one, corresponding to a $L^2(\Omega) \times H^{-2}(\Omega)$ energy. This can be seen from a functional analysis point of view as the plate equation is of the form $\partial_{tt} y + \mathcal{A} y = 0$ for a positive definite self-adjoint operator on a Hilbert space $H$, where $\mathcal{A}$ is given by the operator $\Delta ( a \Delta \cdot)$ with domain $H^2_0(\Omega)$, and thus not only its natural energy $t \mapsto \| \mathcal{A}^{1/2} y(t) \|_{H}^2 + \| \partial_t y(t)\|_H^2$ is constant, but also weaker ones, such as $t \mapsto \| y(t) \|_{H}^2 + \| \partial_t \mathcal{A}^{-1/2} y(t)\|_H^2$. Here, as we should be careful with the metric, we derive it directly using PDE techniques similar as what we did in \eqref{Def-w-n} for the wave equation. More precisely, for a solution $y$ of \eqref{Plate-Eq} with initial data in $H^2_0(\Omega) \times L^2(\Omega)$, we introduce a function $w$ defined by, for all $t \in [0,T]$,  
	$$
		\left\{ \begin{array}{ll}
		\Delta (a \Delta w(t) ) = y(t) &\text{in } \Omega,
			\\
		w(t) = \partial_\nu w(t) = 0 & \text{on } \partial \Omega, 
		\end{array}\right.
	$$ 
	and we multiply the equation \eqref{Plate-Eq} by $\partial_t w$. This yields that the weak energy 
	\begin{equation}
		\label{Energy--1-plate}
		E_{-1}(t) = \frac{1}{2} \int_\Omega \left( | y(t)|^2 + a |\Delta \partial_t w(t)|^2 \right) \, \ud x
	\end{equation}
	is independent of the time variable $t \in [0,T]$.
	\smallskip
	
	Let us now go to the proof of Proposition \ref{Prop-Uniqueness-Plate}. We will only present the details in the case $y = 0$ in $(0,T) \times (\omega \cap \Omega)$, as the other one can be proved using exactly the same arguments. 
	
	We introduce the set 
	$$
		\mathcal{N}(T) = \{ (y_0, y_1) \in H^2_0(\Omega) \times L^2(\Omega), \, \text{ the solution $y$ of \eqref{Plate-Eq} satisfies } y = 0 \text{ in }(0,T)\times (\omega \cap \Omega)\}.
	$$
	Using the observability estimate \eqref{Obs-Plates-Distributed-smooth} with $T_0 = T/3$ and $T_1 = 2T/3$ (recall Proposition \ref{Prop-Obs-Up-To-Comp}), for all $(y_0, y_1) \in H^2_0(\Omega) \times L^2(\Omega)$, 
	\begin{equation}
		\label{Toward-Compactness-0}
		\int_\Omega ( a |\Delta y_0|^2 + |y_1|^2 ) \ud x 
		\leq 
		C_\omega^2 \left( 
		\| \nabla y(T/3)\|_{L^2(\Omega)}^2 + \| \nabla y(2T/3)\|_{L^2(\Omega)}^2\right).
	\end{equation}
	But for all $t \in [0,T]$, 
	$$
		\|\nabla y(t)\|_{L^2(\Omega)}^2
		\leq 
		\| y(t) \|_{L^2(\Omega)} \|\Delta y\|_{L^2(\Omega)}
		\leq
		\sqrt{\alpha} \| y(t) \|_{L^2(\Omega)} \|\sqrt{a} \Delta y\|_{L^2(\Omega)}. 
	$$
	Using thus the fact that the energies %$E(t) = E(0)$ 
    $E$ in \eqref{Energy-Plate} and $E_{-1}$ 
    %(t) = E_{-1}(0)$ 
    in \eqref{Energy--1-plate} are preserved, we obtain
	$$	
		\|\nabla y(t)\|_{L^2(\Omega)}^2
		\leq 
		C (E(0))^{1/2} (E_{-1}(0))^{1/2}. 
	$$
	Therefore, the estimate \eqref{Toward-Compactness-0} implies that for all $(y_0, y_1) \in \mathcal{N}(T)$, $E(0) \leq C E_{-1}(0)$, that is 
	\begin{equation}
		\label{Toward-Compactness-1}
		\frac{1}{\alpha} \int_\Omega (  |\Delta y_0|^2 + |y_1|^2 ) \ud x 
		\leq
		\int_\Omega ( a |\Delta y_0|^2 + |y_1|^2 ) \ud x 
		\leq 
		C 
		\int_\Omega ( |y_0|^2 + a |\Delta w_1|^2 ) \ud x
		\leq 
		C \alpha 
		\int_\Omega ( |y_0|^2 +  |\Delta w_1|^2 ) \ud x, 
	\end{equation}
	where $w_1$ is the solution of 
	\begin{equation}
	    \label{Def-w-1-plate}
		\left\{ \begin{array}{ll}
			\Delta (a \Delta w_1 ) = y_1 &\text{in } \Omega,
			\\
			w_1 = \partial_\nu w_1 =0 & \text{on } \partial \Omega. 
		\end{array}\right.
	\end{equation}
	The estimate \eqref{Toward-Compactness-1} implies that the set $\mathcal{N}(T)$ is closed for the topology given by the norm given by $$\|(y_0, y_1)\|_{\mathcal{N}(T)}^2 = \| y_0\|_{L^2(\Omega)}^2 + \|\Delta w_1 \|_{L^2(\Omega) }^2,$$ with $w_1$ given as above by \eqref{Def-w-1-plate}. 
	
	Furthermore, it implies that the ball of $\mathcal{N}(T)$ endowed with this norm is compact. Indeed, let $(y_{0,n}, y_{1,n})_{n \in \N}$ be a sequence of $\mathcal{N}(T)$ such that $\| (y_{0,n}, y_{1,n}) \|_{\mathcal{N}(T)} = 1$ for all $n \in \N$. Then $(y_{0,n}, w_{1,n})_{n \in \N}$ are bounded in $L^2(\Omega) \times H^2_0(\Omega)$. From \eqref{Toward-Compactness-1}, we get that $(y_{0,n}, y_{1,n})_{n \in \N}$ is bounded in $H^2_0(\Omega) \times L^2(\Omega)$. Thus, up to a subsequence still denoted the same for simplicity, we have that  the sequence $(y_{1,n})_{n \in \N}$ weakly converges to some $y_1 \in L^2(\Omega)$ and, from the compactness of the embedding of $H^2_0(\Omega) $ in $L^2(\Omega)$, $(y_{0,n}, w_{1,n})_{n \in\N}$ converges to $(y_0,w_1)$ strongly in $L^2(\Omega) \times L^2(\Omega)$. Since 
	$$ 
		\int_\Omega a |\Delta (w_{1,n} - w_{1}) |^2 \, \ud x 
		= 
		\int_\Omega  (w_{1,n} -w_1) \Delta (a \Delta (w_{1,n}-w_1) ) \, \ud x 
		= 
		\int_\Omega  (w_{1,n}-w_1) ( y_{1,n} - y_1) \, \ud x,
	$$
	we deduce that the sequence $(y_{0,n}, y_{1,n})_{n \in \N}$ is strongly convergent in $\mathcal{N}(T)$. Consequently, the unit ball of $\mathcal{N}(T)$ is compact, and thus $\mathcal{N}(T)$ is finite dimensional. 

	Let us now take $(y_0, y_1) \in \mathcal{N}(T)$. Remark then that $(y_0, y_1) \in \mathcal{N}(T/2)$. We can then remark that for $\tau \in (0,T/2)$, $(z_{\tau,0}, z_{\tau,1} ) = ((y(\tau) - y_0)/\tau, (\partial_t y(\tau) - y_1)/\tau)$ is a sequence of data in $\mathcal{N}(T/2)$, since the corresponding solution of \eqref{Plate-Eq} is $z_\tau = (y(\cdot + \tau ) - y)/\tau$. From \eqref{Toward-Compactness-1} applied for elements of $\mathcal{N}(T/2)$, we get
	\begin{equation}
		\label{Toward-Compactness-2}
		 \int_\Omega (  |\Delta z_{\tau,0}|^2  + |z_{\tau,1}|^2 ) \ud x 
		\leq 
		C 
		\int_\Omega ( |z_{\tau,0}|^2 + a |\Delta w_{\tau}|^2 ) \ud x, 
	\end{equation}
	where $w_{\tau}$ is the solution of 
	$$
		\left\{ \begin{array}{ll}
			\Delta (a \Delta w_\tau ) = z_{\tau,1} &\text{in } \Omega,
			\\
			w_\tau = \partial_\nu w_\tau =0 & \text{on } \partial \Omega.
		\end{array}\right.
	$$
	Using that 
	$$
		z_{\tau,0} = \int_0^1 \partial_t y(s \tau) \, \ud s
		\text{ and } 
		z_{\tau,1} = \int_0^1 \partial_{tt} y( s \tau ) \, \ud s = - \Delta \left(a \Delta \left( \int_0^1 y(s\tau) \, \ud s\right)\right), 
	$$
	we obtain $w_{\tau} = -\int_0^1 y(s \tau)\, \ud s$ and consequently,
	$$
		\int_\Omega ( |z_{\tau,0}|^2 + a |\Delta w_{\tau}|^2 ) \ud x
		\leq 
		C \int_\Omega ( a| \Delta  y_0|^2 + |y_1|^2) \, \ud x, 
	$$
	for $C$ independent of $\tau$. From \eqref{Toward-Compactness-2}, the family $(z_{\tau,0}, z_{\tau,1})_{\tau \in (0, T/2)}$ is thus bounded in $H^2_0(\Omega) \times L^2(\Omega)$, while it is clear that it weakly converges in $L^2(\Omega) \times H^{-2}(\Omega)$ to $(y_1, -\mathcal{A} y_0)$ as $\tau \to 0$. Hence passing to the limit $\tau \to 0$, we get that $(y_1, -\mathcal{A} y_0 )$ belongs to $H^2_0(\Omega) \times L^2(\Omega)$.
	In other words, the operator 
	$$
		\mathfrak{A} = \left( \begin{array}{ll} 0 & Id \\ - \mathcal{A} & 0 \end{array}\right)
	$$
	acts on $\mathcal{N}(T)$. Since it corresponds to the time differentiation of the solutions of \eqref{Plate-Eq}, $\mathfrak{A}$ maps $\mathcal{N}(T)$ into itself. 
	
	Since $\mathcal{N}(T)$ is finite dimensional, if it is non-trivial, $\mathfrak{A}$ has a non-trivial eigenvector $(y_0, y_1)$ in $\mathcal{N}(T)$ such that the corresponding solution $y$ of \eqref{Plate-Eq} satisfies, for some $\lambda \in \C$, $y(t) = y_0 e^{\lambda t}$ for all $t \in [0,T]$ and $y = 0$ in $(0,T) \times (\omega\cap \Omega)$. If such solution exists, then $y(t) = y_0 e^{\lambda t}$ is a solution of \eqref{Plate-Eq} for all times, that is for $t \in [0, \infty)$, which furthermore satisfies $y = 0$ in $(0,\infty) \times (\omega \cap \Omega)$. From Proposition \ref{Prop-Obs-Up-To-Comp} and the observability estimate \eqref{Obs-Plates-Dis-smooth-large-time} applied to $a \in L^\infty(\Omega)$ (recall that $a$ satisfies the assumptions of Proposition \ref{Prop-Obs-Up-To-Comp}, this would entail $y = 0$ everywhere, that is $y_0 = 0$, and we have reached a contradiction. 
	
	In other words, we have proved that for all $T>0$, $\mathcal{N}(T) = \{0\}$.
\end{proof}

As a corollary, we have:

\begin{corollary}
	\label{Cor-Uniqueness-Plate}
	Under the assumption of Theorem \ref{MainThmOBS-Plates}, if $y$ is a solution of \eqref{Plate-Eq} with initial datum in $H^2_0(\Omega) \times L^2(\Omega)$ satisfying  $\partial_t y= 0$ in $(0,T) \times (\omega \cap \Omega)$, we also have that $y = 0$ in $(0,T) \times \Omega$. 
\end{corollary}

\begin{proof}	
	Let $y$ be a solution of \eqref{Plate-Eq} with initial datum in $H^2_0(\Omega) \times L^2(\Omega)$ satisfying  $\partial_t y= 0$ in $(0,T) \times (\omega \cap \Omega)$. Then for all $\tau \in (0,T/2)$, the function $z_\tau$ defined by $z_\tau (t) = y(t+\tau) - y(t)$ for $t \in (0,T/2)$ is a solution of \eqref{Plate-Eq} with initial datum in $H^2_0(\Omega) \times L^2(\Omega)$ satisfying $z_\tau = 0$ in $(0,T/2) \times (\omega\cap \Omega)$. Using then Proposition \ref{Prop-Uniqueness-Plate} (in the time horizon $T/2$) we have that for all $\tau \in (0,T/2)$, $z_\tau = 0$ in $(0, T/2) \times \Omega$. In particular, passing to the limit $\tau \to 0$, we get $\partial_t y = 0$ in $(0,T/2) \times \Omega$, so that the equation \eqref{Plate-Eq} implies that for all $t \in (0,T/2)$, $\Delta (a \Delta y(t)) = 0$ in $\Omega$ and $y (t)= \partial_\nu y(t)=0$ on $\partial\Omega$. Multiplying this equation by $y(t)$, we immediately deduce that $y(t) = 0$ in $(0,T/2)\times \Omega$. Since we also have that $\partial_t y = 0$ in $(0,T/2) \times \Omega$, it means that the energy at any time $t \in (0,T/2)$ is zero. Since this is a constant quantity, the solution $y$ vanishes everywhere.
\end{proof}

\subsubsection{Proof of Theorem \ref{MainThmOBS-Plates}: a compactness argument} 
 \begin{proof}[Proof of Theorem \ref{MainThmOBS-Plates}.]
 	We simply argue by contradiction as before, starting with the case of an observation on $\Gamma_0$, and the proof of the observability property \eqref{Obs-Plate-Gamma}.
	
	We assume that there exists a sequence $(y_{0,n}, y_{1,n}) \in H^2_0(\Omega) \times L^2(\Omega)$ such that
	\begin{equation}
		\label{Contradiction-Ass-Boundary}
		\forall n \in \N, \quad \int_\Omega \left( a |\Delta y_{0,n}|^2 + |y_{1,n}|^2 \right) \, \ud x  = 1, 
		\text{ and } 
		\lim_{n \to \infty} \int_0^T \int_{\Gamma_0} |\Delta y_n|^2 \, \ud \sigma\,  \ud t = 0, 
	\end{equation}
	where $y_n$ is the solution of \eqref{Plate-Eq} with initial datum $(y_{0,n}, y_{1,n})$. 
	
	Then there exists $(y_0, y_1) \in H^2_0(\Omega) \times L^2(\Omega)$ such that the sequence $(y_{0,n} , y_{1,n})_{n \in \N}$ weakly converges (up to a subsequence still denoted the same) to $(y_0, y_1)$ in $H^2_0(\Omega) \times L^2(\Omega)$. Besides, since the sequence $(y_n)_{n \in \N}$ is bounded in $L^\infty(0,T; H^2_0(\Omega)) \times W^{1,\infty}(0,T; L^2(\Omega))$, it also weakly-$\star$ converges to the solution $y$ of \eqref{Plate-Eq} with initial data $(y_0, y_1)$ in $L^\infty(0,T; H^2_0(\Omega)) \times W^{1,\infty}(0,T; L^2(\Omega))$, and the assumption \eqref{Contradiction-Ass-Boundary} gives that $\Delta y = 0$ on $(0,T) \times \Gamma_0$. It follows from Proposition \ref{Prop-Uniqueness-Plate} that $y$ vanishes in $(0,T) \times \Omega$ and $(y_0, y_1) = (0,0)$.
	
	Now, the sequence $(y_n(T))_{n \in \N}$ weakly converges to $y(T) = 0$ in $H^2_0(\Omega)$. Using the compactness of the embedding of $H^2_0(\Omega)$ into $H^1_0(\Omega)$, we deduce that the sequences $(y_{0,n})_{n \in \N}$ and $(y_n(T))_{n \in \N}$ strongly converges to $y_0 = 0 $ and $y(T)= 0$ in $H^1_0(\Omega)$. The estimate \eqref{Obs-Plates-Boundary-smooth} then contradicts the assumption \eqref{Contradiction-Ass-Boundary}. This concludes the proof of the observability property \eqref{Obs-Plate-Gamma}.
	\smallskip
	
	The proof of the observability property \eqref{Obs-Plate-Dis} follows similarly by contradiction, using the estimate \eqref{Obs-Plates-Distributed-smooth} obtained from Proposition \ref{Prop-Obs-Up-To-Comp} and the uniqueness result from Corollary \ref{Cor-Uniqueness-Plate}, and the compactness embeddings from $H^2_0(\Omega)$ in $L^2(\Omega)$ (Rellich theorem) and from $L^2(0,T; H^2_0(\Omega)) \cap H^1(0,T; L^2(\Omega))$ in $L^2(0,T; H^1_0(\Omega))$ (Aubin-Lions theorem). Details are left to the reader.	%
 \end{proof}

\subsubsection{Some comments on the multiplier type condition for the plate equation}

In this paragraph, we would like to underline that the multiplier condition \eqref{Multiplier-Condition-Plate} is also quite natural from the point of view of microlocal analysis. 

Indeed, to understand the propagation of the wave front sets (or microlocal defect measures) for operators of  the form $\partial_{tt} + \Delta( a \Delta \cdot)$ away from the boundary and for smooth coefficients $a$, we should consider the principal symbol of the operator, namely $p(x,\xi)=a(x)|\xi|^{4}$ and its Hamiltonian flow given in this case by the differential system:
\begin{equation*}
    \frac{dx}{dt}=-4 |\xi|^{2}\xi a(x),\qquad
    \frac{d\xi}{dt}=\nabla_x a(x)\, |\xi|^{4}, 
\end{equation*}
when the initial datum $(x_0, \xi_0)$ satisfies $a(x_0) |\xi_0|^4 = 1$.

It follows from the equations that 
\begin{equation}\label{eq:hf1}
\frac{d}{dt}\left(\frac{1}{2}|x|^2\right)=x\cdot\frac{dx}{dt}
= - 4 (x\cdot\xi)a(x)|\xi|^{2} = - 4 (x \cdot \xi) \sqrt{a(x)}
\end{equation}
as along bicharacteristics, $(x(t),\xi(t))$ satisfies $a(x(t)) |\xi(t)|^4 = 1$ for all $t$.

Now, we also have 
\begin{equation}\label{eq:hf2}
	\frac{d}{dt}(x\cdot\xi)=\frac{dx}{dt}\cdot\xi+x\cdot\frac{d\xi}{dt}
=|\xi|^{4}(-4 a(x)+x\cdot\nabla_x\, a(x)). 
\end{equation}
From \eqref{eq:hf1} and \eqref{eq:hf2} we obtain
\begin{equation*}
\frac{d}{dt}\left(\frac{1}{\sqrt{a(x)}} \frac{d}{dt}\left(\frac{|x|^2}{2}\right)\right)
= \frac{4}{a(x)} (4\,a(x)-x\cdot\nabla_x\, a(x)).
\end{equation*}
It follows that, if we are in $\R^d$, the condition \eqref{Multiplier-Condition-Plate} in fact guarantees that any trajectory $t \mapsto (x(t), \xi(t))$ satisfies that $t \mapsto |x(t)|^2$ will go to infinity as $t \to \infty$. 

Also, if there exists a sphere $\partial B (0,R)$ such that $x \cdot \nabla a = 4 a$ on $\partial B (0,R)$, then taking $x_0 \in \partial B (0,R)$ and $\xi_0 \in \R^d$ with $\xi_0 \cdot x_0 = 0$ and $|\xi_0| = a(x_0)^{-1/4}$, for all $t \geq 0$, $x(t) \cdot \xi(t) = 0$, and the trajectory described by the Hamiltonian flow thus stays on the sphere $\partial B(0,R)$. Although this does not necessarily imply that the plate equation \eqref{Plate-Eq} in a domain containing this sphere would be not observable, such ray would certainly require additional work to be dealt with, even in the context of smooth coefficients, using for instance second micro-localization techniques (see \cite{Anantharaman-Leautaud}).

\bibliography{bibliography.bib}

\end{document}